\documentclass[11 pt]{amsart}
\usepackage{amsfonts}
\usepackage{amsmath,amscd}
\usepackage{fullpage}
\usepackage{amssymb}
\usepackage{centernot} 
\usepackage{enumerate} 
\usepackage{parskip}
\usepackage{pb-diagram} 
\usepackage{mathrsfs}
\usepackage[OT2,T1]{fontenc}
\usepackage{seqsplit}
\usepackage{color}
\usepackage{array}
\usepackage{verbatim}
\usepackage{url}

\def\IZ{\text{\rm I$_0$}}
\def\IZS{\text{\rm I$_0^*$}}
\def\In#1{{\text{\rm I{}$_{#1}$}}}
\def\InS#1{{\text{\rm I{}$_{#1}^*$}}}
\def\II{\text{\rm II}}
\def\IIS{\text{\rm II$^*$}}
\def\III{\text{\rm III}}
\def\IIIS{\text{\rm III$^*$}}
\def\IV{\text{\rm IV}}
\def\IVS{\text{\rm IV$^*$}}

\DeclareSymbolFont{cyrletters}{OT2}{wncyr}{m}{n}
\DeclareMathSymbol{\Sha}{\mathalpha}{cyrletters}{"58}

\definecolor{refkey}{rgb}{1,1,1}
\definecolor{labelkey}{rgb}{1,1,1}
\definecolor{cite}{rgb}{0.9451,0.2706,0.4941}
\definecolor{ruri}{rgb}{0.0078,0.4022,0.8010}

\usepackage[%
bookmarks=true,bookmarksnumbered=true,%
colorlinks=true,linkcolor=ruri,citecolor=red%
 ]{hyperref}

\makeindex 

\def\F{{\rm \mathbb{F}}}
\def\Z{{\rm \mathbb{Z}}}

\def\Q{{\rm \mathbb{Q}}}

\def\C{{\rm \mathbb{C}}}

\def\R{{\rm \mathbb{R}}}

\def\P{{\rm \mathbb{P}}}

\def\p{{\rm \mathfrak{p}}}
\def\m{{\rm \mathfrak{m}}}

\def\q{{\rm \mathfrak{q}}}

\def\O{{\rm \mathcal{O}}}

\def\A{{\rm \mathbb{A}}}

\def\Nm{{\rm Nm}}

\def\PGL{{\rm PGL}}

\def\irr{{\rm irr}}
\def\minn{{\rm min}}
\def\ord{{\rm ord}}

\def\avg{{\rm avg}}

\def\inv{{\rm inv}}
\def\Aut{{\rm Aut}}
\def\la{{\rm \longrightarrow \,}}

\def\Tr{{\rm Tr}}

\def\Disc{{\rm Disc}}
\def\disc{{\rm disc}}

\def\SL{{\rm SL}}

\def\Res{{\rm Res}}
\def\Sym{{\rm Sym}}
\def\sol{{\phi{\text{-}{\rm sol}}}}
\def\sols{{\phi_s{\text{-}{\rm sol}}}}

\def\GL{{\rm GL}}
\def\Gal{{\rm Gal}}

\def\End{{\rm End}}
\def\Sel{{\rm Sel}}

\def\ls{{{\rm loc.}\hspace{.35mm}\phi\text{-}{\rm sol.}}}

\def\coker{{\rm coker \hspace{1mm}}}

\numberwithin{equation}{section}

\newtheorem{theorem}{Theorem}[section]
\newtheorem{lemma}[theorem]{Lemma}

\newtheorem{remark}[theorem]{Remark}
\newtheorem{definition}[theorem]{Definition}
\newtheorem{example}[theorem]{Example}

\newtheorem{corollary}[theorem]{Corollary}
\newtheorem{proposition}[theorem]{Proposition}

\usepackage{marginnote,color}
\usepackage[usenames,dvipsnames]{xcolor}

\usepackage{tikz}

\def\shownotes{\def\inline##1##2##3{ \begin{adjustwidth}{3mm}{7mm}\mbox{}\par \noindent
{\color{##1}\hspace{-1.9cm}{\large ##2}\vspace{-\baselineskip}\\##3}
\newline\end{adjustwidth}} \def\inlinewide##1##2##3{ \begin{adjustwidth}{0mm}{0cm}\mbox{}\par \noindent
{\color{##1}\hspace{-1.6cm}{\large ##2}\vspace{-\baselineskip}\\##3}
\newline\end{adjustwidth}}  \def\marg##1##2##3{\marginnote{\color{##1}{\large ##2}\\{\small ##3}}[-.8cm]}}

\shownotes

\begin{document}
\setlength{\parskip}{2pt} 
\setlength{\parindent}{8pt}
\title{Three-isogeny Selmer groups and ranks of abelian varieties in quadratic twist families over a number field}
\author{Manjul Bhargava}
\address{Department of Mathematics, Princeton University, Princeton, NJ 08544}
\email{bhargava@math.princeton.edu}

\author{Zev Klagsbrun}
\address{Center for Communications Research, San Diego, CA 92121}
\email{zdklags@ccrwest.org}

\author{Robert J. Lemke Oliver}
\address{Department of Mathematics, Tufts University, Medford, MA 02155}
\email{robert.lemke{\_{}}oliver@tufts.edu}

\author{Ari Shnidman}
\address{Department of Mathematics, Boston College, Chestnut Hill, MA 02467} 
\email{shnidman@bc.edu}

\maketitle

\begin{abstract}
For an abelian variety $A$ over a number field $F$, we prove that the average rank of the quadratic twists of $A$ is bounded, under the assumption that the multiplication-by-3 isogeny on~$A$ factors as a composition of 3-isogenies over~$F$.    
This is the first such boundedness result for an absolutely simple abelian variety $A$ of dimension greater than one.  In fact, we exhibit such twist families in arbitrarily large dimension and over any number field.  

In dimension one, we deduce that if $E/F$ is an elliptic curve admitting a 3-isogeny, then the average rank of its quadratic twists is bounded. If $F$ is totally real, we moreover show that a positive proportion of twists have rank 0 and a positive proportion have $3$-Selmer rank 1.  
These results on bounded average ranks in families of quadratic twists represent new progress towards Goldfeld's conjecture -- which states that the average rank in the quadratic twist family of an elliptic curve over $\Q$ should be $1/2$ -- and the first progress towards the analogous conjecture over number fields other than $\Q$.  


Our results follow from a computation of the average size of the $\phi$-Selmer group in the family of quadratic twists of an abelian variety admitting a 3-isogeny $\phi$.  


\end{abstract}
\vspace{.5cm}

\makeatletter
\makeatother

\section{Introduction}

Let $A$ be an abelian variety over a number field $F$.  For any squareclass $s \in F^*/F^{*2}$, the quadratic twist $A_s$ is another abelian variety over~$F$ that becomes isomorphic to $A$ upon base change to $F(\sqrt s)$ (see Section \ref{soluble orbits}). This recovers the usual notion of quadratic twist for elliptic curves: if $A$ is an elliptic curve with Weierstrass model $y^2 = f(x)$, then $A_s$ has model $\tilde sy^2 = f(x)$, for any lift $\tilde s$ of $s$ to $F^*$.  

Quadratic twist families are the simplest possible families of abelian varieties---geometrically, they correspond to a single point in the moduli space.  Nevertheless, we know very little about the arithmetic of such families.  For example, for general abelian varieties $A$, we know essentially nothing about the behavior of the ranks of the Mordell--Weil groups $A_s(F)$ as $s$ varies.  In the special case where $A = E$ is an elliptic curve over $\Q$, Goldfeld conjectured that the average rank of $E_s$ is equal to $1/2$, with $50\%$ of twists having rank 0 and $50\%$ having rank 1 \cite{Goldfeld}.  Even in this case, very little is known.  Smith \cite{Smith}, building on work of Kane \cite{Kane}, has recently proved that in the very special case where $E[2](\Q) = (\Z/2\Z)^2$, the average rank is at most $1/2$.  To date, this is the only case where the average rank of $A_s(F)$ has been proven to be bounded.

In this paper, we prove the boundedness of the average rank of $A_s(F)$ for a large class of abelian varieties $A$, namely,  those with an isogeny $A \to A$ that factors as a composition of 3-isogenies over~$F$.  As a consequence, over $\Q$, we give the first known examples of elliptic curves $E$ not having full rational 2-torsion for which the average rank of $E_s(\Q)$ is bounded, specifically, those $E$ having a rational subgroup of order 3.  Over general number fields $F \neq \Q$, we give the first known examples of elliptic curves of any kind for which the average rank of $E_s(F)$ is bounded.  We also provide the first known examples of absolutely simple abelian varieties $A$ over $\Q$ having dimension greater than one for which the average rank of $A_s(\Q)$ is bounded, and indeed we give such examples over any number field $F$ and in arbitrarily large dimension.  

The proofs of these results are accomplished through a computation of the average size of a family of Selmer groups, namely, the $\phi$-Selmer groups associated to a quadratic twist family of 3-isogenies $\phi_s \colon A_s \to A'_s$ of abelian varieties.  There have been a number of recent results on the average sizes of Selmer groups in large universal families of elliptic curves or Jacobians of higher genus curves (see, e.g., \cite{BG,BH, BS1,BS2,SW, Ananth,Th}) obtained via geometry-of-numbers methods. It has been a recurring question as to whether such methods could somehow be adapted to obtain analogous results in much thinner (e.g.,\ one-parameter) families.  The current work represents the first example of such methods being used to determine the average size of a Selmer group in a family of quadratic twists.  

\section{Results}

     
We now describe our results more precisely.  Suppose $A$ admits an isogeny $\phi \colon A \to A'$ over~$F$.  One can define a Selmer group $\Sel_{\phi}(A)$, which is a direct generalization of  the $n$-Selmer group $\Sel_n(E)$ attached to the multiplication-by-$n$ isogeny $[n] \colon E\to E$ on an elliptic curve (see Section \ref{integerorbits}).  For each $s \in F^*/F^{*2}$, there is also an isogeny $\phi_s \colon A_s \to A'_s$ and a corresponding Selmer group $\Sel_{\phi_s}(A_s)$.  Our main theorem computes the average size of $\Sel_{\phi_s}(A_s)$ as $s$ varies, in the case where $\phi$ has degree~3.          

For any place $\p$ of $F$, we write $F_\p$ for the $\p$-adic completion of $F$.  From the results of \cite{j=0}, one naturally expects that the average number of non-trivial elements of $\Sel_\phi(A_s)$ is governed by an Euler product whose factor at $\p$ is the average size of the {\it local $\phi$-Selmer ratio}
\[c_\p(\phi_s) := \dfrac{|A_s'(F_\p)/\phi(A_s(F_\p))|}{|A_s[\phi](F_\p)|},\]
as $s$ varies over $F_\p^*/F_\p^{*2}$. 
This is exactly what we prove, but unlike in \cite[Thm.\ 2]{j=0}, the Euler product turns out to be finite.  
This allows us to formulate our results as follows.  

We call a subset $\Sigma_\p$ of $F_\p^*/F_\p^{*2}$ a {\it local condition}.  If $\p$ is infinite, then we give $F_\p^*/F_\p^{*2}$ the discrete measure, and if $\p$ is finite, we give $F_\p^*/F_\p^{*2}$ the measure obtained by taking the usual (normalized) Haar measure of  the preimage of $\Sigma_\p$ in $\O_{F_\p}$.  
This choice of measure gives rise to a natural notion of the average of $c_\p(\phi_s)$ over $s\in \Sigma_\p$, which we denote by $\mathrm{avg}_{\Sigma_\p} c_\p(\phi_s)$.  In fact, it turns out that $c_\p(\phi_s)=1$, regardless of $s$, for every finite prime $\p \nmid 3 \mathfrak{f}_A$, where $\mathfrak{f}_A$ is the conductor of $A$ (Theorem~\ref{integrality}). Hence this local average is equal to $1$ for all but finitely many primes.

Recall the definition of the {\it global Selmer ratio}:
\[c(\phi_s) := \prod_{\p \leq \infty} c_\p(\phi_s).\]
We say that a subset $\Sigma \subset F^*/F^{*2}$ is defined by local conditions if $\Sigma = F^*/F^{*2} \cap \prod_\p \Sigma_\p$, where each $\Sigma_\p\subseteq F_\p^*/F_\p^{*2}$ is a local condition; if $\Sigma_\p = F_\p^*/F_\p^{*2}$ for all but finitely many $\p$, then we say that $\Sigma$ is defined by finitely many local conditions.  We define the height $H(s)$ of $s\in F^*/F^{*2}$ by 
$$H(s) := \prod_{\mathfrak{p}\,:\,v_\mathfrak{p}(s){\rm  \;is\; odd}}N(\mathfrak{p}).$$
For any subset $\Sigma\subset F^*/F^{*2}$ defined by finitely many local conditions, we may then define the global average of $c(\phi_s)$ over all $s\in\Sigma$ by
\[
\mathrm{avg}_\Sigma c(\phi_s) := \lim_{X\to\infty} \frac{1}{|\Sigma(X)|}\displaystyle\sum_{s \in \Sigma(X)} c(\phi_s) = \prod_{\p \mid 3\mathfrak{f}_A \infty} \mathrm{avg}_{\Sigma_\p} c_\p(\phi_s), 
\]
which is a rational number; here, $\Sigma(X) := \{ s \in \Sigma \colon H(s) < X\}$.  If $F=\Q$, then $\Sigma(X)$ consists of the squareclasses of all squarefree integers of absolute value less than $X$, recovering the usual ordering of quadratic twists over $\Q$.

Our main result is:
\begin{theorem}\label{main}
Let $\phi \colon A \to A'$ be a $3$-isogeny of abelian varieties over a number field $F$, and let $\Sigma$ be a non-empty subset of $F^*/F^{*2}$ defined by finitely many local conditions.  When the abelian varieties $A_s$, $s \in \Sigma$, are ordered by the height of $s$, the average size of 
$\Sel_{\phi_s}(A_s)$ is $1 + \avg_\Sigma \, c(\phi_s)$.
\end{theorem}


Theorem \ref{main} has a number of applications to ranks of quadratic twists of abelian varieties.

\begin{theorem}\label{ab var ranks}
Suppose $A$ is an abelian variety over a number field $F$ with an isogeny $A \to A$ that factors as a composition of $3$-isogenies over~$F$.  Then the average rank of $A_s(F)$, $s \in F^*/F^{*2}$, is bounded. 
\end{theorem}

Theorem \ref{ab var ranks} gives the first known examples of absolutely simple abelian varieties of dimension greater than one whose quadratic twists have bounded rank on average.  
As an example, Theorem~\ref{ab var ranks} applies to the Jacobian $J(C)$ of the genus three curve $C \colon y^3 = x^4-x$ over $\mathbb{Q}$, or indeed over any number field $F$.  For a robust set of examples of abelian surfaces over~$F$ to which Theorem~\ref{ab var ranks} applies, see \cite{BFT} and \cite{BN}.  In Section~\ref{superelliptic}, we provide a large class of such examples of absolutely simple Jacobians of arbitrarily large dimension by considering the trigonal curves $y^3 = f(x)$. 

Theorem \ref{main} also yields the first known examples of absolutely simple abelian varieties over a number field $F$ of dimension greater than one with a positive proportion of twists having rank 0.  One example is the base change of $J(C)$ above to the cyclotomic field $\Q(\zeta_9)$, for which we prove that at least~$50\%$ of twists have rank 0 (Theorem \ref{picard50}).  
  
Any abelian variety obtains full level-three structure upon base change to a sufficiently large number field, so that the multiplication-by-three isogeny is the composition of 3-isogenies.  We thus obtain the following immediate corollary of Theorem \ref{ab var ranks}, which is new even in the case that $A$ is an elliptic curve.          

\begin{theorem}
Let $A$ be an abelian variety over a number field $F$.  Then there exists a finite field extension $L/F$ such that the average rank of the quadratic twist family $A_{L,s}$, $s \in L^*/L^{*2}$, of the base change $A_L$ is bounded.
\end{theorem}

When $A$ is an elliptic curve, the hypothesis in Theorem \ref{ab var ranks} is simply that $A$ admits a 3-isogeny, or equivalently, an $F$-rational subgroup of order 3.  Accordingly, we can be much more precise in this case.   
To state the result, we introduce the {\em logarithmic Selmer ratio} $t(\phi_s)$: the global Selmer ratio $c(\phi_s)$ lies in $3^\Z$, and we take $t(\phi_s) := \ord_3 \, c(\phi_s)$.
   
\begin{theorem}\label{avgrank}
Suppose $E/F$ is an elliptic curve admitting a $3$-isogeny and let $\Sigma \subset F^*/F^{*2}$ be defined by finitely many local conditions.  Then the average rank of $E_s(F)$, $s \in \Sigma$, is at most  $\avg_\Sigma \left( |t(\phi_s)| + 3^{-|t(\phi_s)|}\right)$.  
\end{theorem}      

Theorem \ref{avgrank} yields the first examples of elliptic curves over $\mathbb{Q}$ that do not have full rational two-torsion whose quadratic twists are known to have bounded average rank.  The case of elliptic curves over $\mathbb{Q}$ with full rational two-torsion is the aforementioned work of Kane~\cite{Kane}, which builds on prior work of Heath-Brown \cite{Heath-Brown} and Swinnerton-Dyer \cite{sd}.

Theorem \ref{avgrank} also yields the first examples of elliptic curves over number fields other than $\Q$ whose quadratic twists are known to have bounded average rank.  While Goldfeld's conjecture \cite{Goldfeld} on average ranks is only stated for quadratic twists of elliptic curves over $\mathbb{Q}$, it has a natural generalization to elliptic curves over number fields (for example, see \cite[Conjecture 7.12]{KMRann} and the ensuing discussion).  Theorem \ref{avgrank} provides the first progress toward this general version of Goldfeld's conjecture.  We note that when the twists $E_s$ are ordered in a non-standard way (essentially by the number of prime factors of $s$ and the largest such factor), the boundedness of average rank has been demonstrated by Klagsbrun, Mazur, and Rubin \cite{KMR} when $\mathrm{Gal}(F(E[2])/F) \simeq S_3$.

The exact value of the upper bound in Theorem \ref{avgrank} depends on the primes for which $E$ has bad reduction and can be computed explicitly for any given curve \cite{sagecode}.  For a ``typical'' elliptic curve with a 3-isogeny, the bound on the average 3-Selmer rank, and thus average rank, given in Theorem \ref{avgrank} is roughly on the order of $\sqrt{\log\log N_E}$, where $N_E$ is the norm of the conductor of $E$.  In particular, the upper bound on the average rank in Theorem~\ref{avgrank} can become arbitrarily large as the curves considered become more complicated.  

Nonetheless, we are still able to prove that there are many twists of small rank in all but the most pathological families.  Define, for each $m \geq 0$, the subset  
\[T_m(\phi) := \{s \in F^*/F^{*2} \colon |t(\phi_s)| = m \},\]
and let $\mu(T_m(\phi))$ denote the density of $T_m(\phi)$ within $F^*/F^{*2}$.

\begin{theorem}\label{globalbounds}
Let $E$ be an elliptic curve over~$F$ admitting a $3$-isogeny $\phi \colon E \to E'$.  Then 
\begin{enumerate}[$($a$)$]
\item The proportion of twists $E_s$ having rank $0$ is at least $\frac{1}{2}\mu(T_0(\phi))$; and
\item The proportion of twists $E_s$ having $3$-Selmer rank $1$ is at least $\frac{5}{6}\mu(T_1(\phi))$.    
\end{enumerate}
\end{theorem}
If we assume the conjecture that $\dim_{\F_3}\Sha(E_s)[3]$ is even for all $s$, then part (b) implies that a proportion of at least $\frac56\mu_1(\phi)$ twists have rank $1$.   

For all but a pathological set of 3-isogenies $\phi$, the sets $T_0(\phi)$ and $T_1(\phi)$ are non-empty and hence have positive density.  In fact, over many fields, we show that these pathologies do not arise at all.

\begin{theorem}\label{totreal}
Let $E$ be an elliptic curve over~$F$ and suppose $E$ admits a $3$-isogeny.
\begin{enumerate}[$(i)$] 
\item If $F$ is totally real, then a positive proportion of $E_s$ have rank $0$ and a positive proportion of $E_s$ have $3$-Selmer rank $1$.
\item If $F$ has at most one complex place and at least one real place, then a positive proportion of $E_s$ have $3$-Selmer rank $1$.   
\item If $F$ is imaginary quadratic, then a positive proportion of $E_s$ have rank $0$ or $1$.
\end{enumerate}      
Furthermore, if $F$ has $r_2$ complex places, then there exists an elliptic curve $E$ over~$F$ admitting a $3$-isogeny and for which $\dim_{\mathbb{F}_3}\mathrm{Sel}_3(E_s) \geq r_2$ for all $s\in F^*/F^{*2}$.
\end{theorem}

Recently and independently, for elliptic curves over $\mathbb{Q}$ with a $3$-isogeny, Kriz and Li \cite{krizli} have given a different proof of the fact that a positive proportion of twists $E_s$ have rank 0 and rank 1.  In fact, their beautiful results in the rank 1 case are unconditional, but the lower bounds on the proportions of rank 0 and rank 1 are smaller than ours, and their method does not give a bound on the {\it average} rank. Previously, there have been a number of special cases of such rank 0 and rank 1 results proved by various authors, using both analytic and algebraic methods; see \cite{chang,James,kriz,zkli,Vatsal,Wang}.     

The proportions provided by Theorem \ref{globalbounds} will naturally be largest when every $s \in F^*/F^{*2}$ is either in $T_0(\phi)$ or $T_1(\phi)$.  The elliptic curve of smallest conductor over $\mathbb{Q}$ for which this occurs is the curve $E\colon y^2 + y = x^3 + x^2 + x$ having Cremona label 19a3.  
Since the parity of $t(\phi_s)$ is easily seen to be equidistributed in quadratic twist families over $\Q$, we deduce that, in this special situation, half of the squareclasses lie in $T_0(\phi)$ and half lie in $T_1(\phi)$.  Thus, by Theorem \ref{avgrank}, the average rank of $E_s$ is at most 7/6, and, by Theorem \ref{globalbounds}, at least $25\%$ of twists $E_s$ have rank 0 and at least $41.6\%$ have $3$-Selmer rank 1.  

Over general number fields, the parity of $t(\phi_s)$ is not necessarily equidistributed in quadratic twist families, and so we can obtain larger proportions of rank 0 or $3$-Selmer rank 1 curves in certain families.   For instance, suppose $E$ is an elliptic curve over a number field $F$ that has complex multiplication by an imaginary quadratic field $K$ contained in $F$, and that $3$ is not inert in~$K$.  Then the curve $E$ is isogenous to a curve $E_0$ with $\End( E_0) \simeq \O_K$, and $E_0$ admits a 3-isogeny $\phi$ to $E_0/E_0[\p]$, where $\p \subset \O_K$ is an ideal of norm 3.  We show in Section \ref{sec:appl-to-rank} that every $s\in F^*/F^{*2}$ then lies in $T_0(\phi)$ for this isogeny $\phi$.  By Theorem \ref{globalbounds}, we thus obtain:

\begin{theorem}\label{cm}
Suppose that $E$ is an elliptic curve over a number field $F$ such that $\End_F(E) \otimes \Q$ is an imaginary quadratic field in which $3$ is not inert.  Then the average rank of $E_s(F)$ for $s \in F^*/F^{*2}$ is at most $1$, and at least $50\%$ of twists $E_s$ have rank $0$.   
\end{theorem}  

In the setting of Theorem \ref{cm}, Goldfeld's minimalist philosophy \cite{Goldfeld} would predict that $100\%$ of twists have rank $0$ since they all have even parity.  In Section \ref{examples}, we give an example of a non-CM elliptic curve with a $3$-isogeny whose twists all have even parity and for which we derive the same conclusion as Theorem \ref{cm}.  We also provide an example of a curve whose twists all have odd parity and for which we may deduce that $5/6$ of twists have $3$-Selmer rank $1$, which we take as strong evidence in favor of the conjecture that $100\%$ of twists of this curve should have rank $1$.


Finally, we note that there are a number of other applications of Theorem \ref{main}, e.g., to the existence of elements of various orders in Tate--Shafarevich groups, which we pursue in forthcoming work.

\subsection*{Methods and organization}

In Sections 2--6, we develop a new correspondence between $\phi$-Selmer elements (in fact, arbitrary $\phi$-coverings) and certain binary cubic forms with coefficients in the base field $F$.  This bijection recovers work of Selmer \cite{selmer} and Satg\'e \cite{satge} in the very special case when $A$ is an elliptic curve over $\Q$ with $j$-invariant 0.  Our correspondence is new for elliptic curves with $j \neq 0$, and for higher dimensional abelian varieties having a 3-isogeny $\phi$.  In Section \ref{integerorbits}, we show, moreover, that integral models exist for $\phi$-Selmer elements, i.e., $\phi$-Selmer elements correspond to integral binary cubic forms.  This may be thought of as a generalization of the $j = 0$ case proven by the first and last authors and Elkies in \cite{j=0}, though the general case treated here indeed exhibits many important differences and subtleties in the `minimization' results.  

Our parametrization of $\phi$-Selmer elements by integral binary cubic forms has the property that $\phi_s$-Selmer elements for different $s\in F^*/F^{*2}$ yield integral binary cubic forms having discriminants that lie in different classes in $F^*/F^{*2}$ as well; moreover, we prove that these discriminants are {\it squarefree away from a fixed finite set of primes}.  This property is what allows us to use results on counting binary cubic forms to get a handle on the count of $3$-isogeny Selmer elements across a family of quadratic twists.  Indeed, in Section \ref{counting}, we combine this new parametrization of $\phi$-Selmer elements with the geometry-of-numbers techniques developed in recent work of the first author, Shankar, and Wang \cite{BSW2}, which enables one to count integral binary cubic forms over a number field satisfying properties of the type described.  This yields Theorem~\ref{main}.  We prove Theorem \ref{ab var ranks} in Section \ref{sec:abvarranks} as a consequence, and also prove Theorems \ref{avgrank} and \ref{globalbounds}.  In Sections \ref{local computations}--\ref{sec:appl-to-rank}, we consider some further applications to ranks of elliptic curves and, in particular, prove Theorems \ref{totreal} and \ref{cm}. In Section \ref{superelliptic}, we show that the Jacobians of a large family of trigonal curves fall under the scope of Theorem \ref{ab var ranks}.  Finally, we use Section \ref{examples} to highlight some examples where our results yield particularly interesting conclusions.

\section{Orbits of binary cubic forms over a Dedekind domain}

We recall from \cite[\S2]{j=0} some facts about orbits of binary cubic forms.  Let $V(\Z) = \Sym_3 \Z^2$ be the lattice of integer-matrix binary cubic forms, i.e., forms 
\[f(x,y) = ax^3 + 3bx^2y + 3cxy^2 + dy^3\] with $a,b,c,d\in\Z$.  Equivalently, $V(\Z)$ is the space of integer symmetric trilinear forms.  The group $\GL_2(\Z)$ acts naturally on $V(\Z)$ by linear change of variable, and we define the {\it discriminant} by
\[\Disc(f) := a^2d^2 -3b^2c^2 +4ac^3 + 4b^3d -6abcd.\]
For any ring $R$, we write $V(R) := V(\Z) \otimes R$.  The action of $\GL_2(R)$ on $V(R)$ satisfies
\[\Disc(g\cdot f) = \det(g)^6 \Disc(f),\]
for all $g \in \GL_2(R)$ and $f \in V(R)$. In particular, the discriminant is $\SL_2(R)$-invariant.  For any $\Delta \in R$, we write $V(R)_\Delta$ for the set of $f \in V(R)$ with $\Disc(f) = \Delta$.  

Now let $R$ be a Dedekind domain of characteristic not 2 or 3, and let $F$ be its field of fractions.  For any non-zero $D \in R$, we define the ring $S = R[t]/(t^2 - D)$ and let $K = F[t]/(t^2 - D)$ be its total ring of fractions.  

\begin{theorem}\cite[Thm.\ 12]{j=0}\label{cubicbij}
Let $D \in R$ be non-zero.  Then there is a bijection between the orbits of $\SL_2(R)$ on $V(R)_{4D}$ and equivalence classes of triples $(I, \delta, \alpha)$, where $I$ is a fractional $S$-ideal, $\delta \in K^*$, and  $\alpha \in F^*$, satisfying the relations $I^3 \subset \delta S$, $N(I)$ is the principal fractional ideal $\alpha R$ in $F$, and $N(\delta) = \alpha^3$ in $F^*$.  Two triples $(I, \delta, \alpha)$ and $(I', \delta', \alpha')$ are equivalent if there exists $\kappa \in K^*$ such that $I '=  \kappa I$, $\delta' = \kappa^3 \delta$, and $\alpha' = N(\kappa)\alpha$.  Under this correspondence, the stabilizer in $\SL_2(R)$ of $f \in V(R)_{4D}$ is isomorphic to $S(I)^*[3]_{N = 1}$, where $S(I)$ is the ring of endomorphisms of $I$. 
\end{theorem}

\begin{definition}{\em 
We call triples $(I, \delta, \alpha)$ satisfying the above conditions {\it $S$-triples}.}
\end{definition}

When $R$ is a field, so that $R = F$, the previous result simplifies quite a bit. We write $(K^*/K^{*3})_{N =1}$ to denote the kernel of the norm map $K^*/K^{*3} \to F^*/F^{*3}$.  We also write $(\Res^K_F  \mu_3)_{N = 1}$ for the kernel of the norm map $\Res^K_F \, \mu_3 \to \mu_3$, where $\Res^K_F\, \mu_3$ is the restriction of scalars of the group scheme $\mu_3$.  

\begin{corollary}\label{bij}
There is a bijection between $\left(K^*/K^{*3}\right)_{N =1}$ and the set of $\SL_2(F)$-orbits of $V(F)_{4D}$.  Moreover, the stabilizer inside $\SL_2(F)$ of any $f \in V(F)_{4D}$ is isomorphic to $(\Res_F^K \, \mu_3)_{N = 1}$.     
\end{corollary}

\begin{proof}
The bijection sends the class of $\delta$ in $\left(K^*/K^{*3}\right)_{N =1}$ to the orbit of the binary cubic form corresponding to the $S$-triple $(K, \delta, \alpha)$, where $\alpha$ is any choice of cube root of $N(\delta)$.  Explicitly, this is the orbit of the cubic form $\frac{1}{2D}\Tr_{K/F}(\delta \tau X^3)$ on $K$, where $\tau$ is the image of $t$ in $K$ \cite[Thm.\ 12]{j=0}.   
\end{proof}

\begin{remark}\label{twisted}{\em
Since $D$ can be replaced by $Ds^2$ for any $s \in F^\times$, the $\SL_2(F)$-orbits of discriminant $4 D$ are in bijection with the $\SL_2(F)$-orbits of discriminant $4 Ds^2$.  The bijection is given by  
$(g, f) \mapsto \frac{1}{\det g} g \cdot f,$ 
where 
$g = \left[\begin{smallmatrix}
 1 & 0\\ 0&s
\end{smallmatrix}\right].$  We call the action $(g, f) \mapsto \frac{1}{\det g} g \cdot f$ of $\GL_2(F)$ on $V(F)$ the {\it twisted} $\GL_2$-action.}
\end{remark}

If $R$ is a discrete valuation ring and if $S$ is the maximal order in $K$, we obtain:

\begin{proposition}\label{intorb}
Suppose $R$ is a discrete valuation ring with fraction field $F$.  Assume that the $R$-algebra $S = R[t]/(t^2 - D)$ is the maximal order in $K = F[t]/(t^2 - D)$.  Then the set of $\SL_2(R)$-orbits on $V(R)_{4D}$ is in bijection with the unit subgroup $(S^*/S^{*3})_{N =1} \subset (K^*/K^{*3})_{N = 1}$.  In particular, every rational $\SL_2(F)$-orbit of discriminant $4D$ whose class $\delta$ lies in this unit subgroup contains a unique integral $\SL_2(R)$-orbit.    
\end{proposition}

\section{Three-isogenies and orbits of binary cubic forms over a field}

Again assume that $F$ is a field of characteristic not 2 or 3, and fix a 3-isogeny $\phi \colon A \to A'$ of abelian varieties over~$F$.  In this section, we relate the $\SL_2(F)$-orbits of binary cubic forms over~$F$ to the arithmetic of $\phi$.  After outlining the general theory, we make more explicit the case of elliptic curves.  

Write $A[\phi]$ for the group scheme $\ker \phi$ over~$F$, and write $G_F = \Gal(\bar F/F)$ for the Galois group of the separable closure of $F$.  As the group $A[\phi](\bar F)$ has order 3, there is an \'etale quadratic $F$-algebra~$K_0$ such that the action of $G_F$ on $A[\phi](\bar F)$ factors through $\Gal(K_0/F)$.  If the action is non-trivial, this $K_0$ is unique.  If the action is trivial, then we take $K_0$ to be the split algebra $F^2$.  We may write $K_0 = F[t]/(t^2 - D)$ for some $D \in F^*$.  Now set $K = F[t]/(t^2  - \hat D)$, with $\hat D = -3D$, and note that~$\hat D$ is only determined by $\phi$ up to squares in $F^*$.  We call $K$ the `mirror algebra' of $A[\phi]$.
The relevance of this mirror comes from the following proposition.    

\begin{proposition}\label{cassels}
There is an isomorphism of  group schemes
\[A[\phi] \simeq  \ker\left(\Res^K_F  \mu_3 \stackrel{\Nm}{\longrightarrow} \mu_3\right),\]  
and hence an induced isomorphism 
\[H^1(G_F, A[\phi]) \simeq \left(K^*/K^{*3}\right)_{N = 1},\]
where $\left(K^*/K^{*3}\right)_{N= 1}$  denotes the kernel of the norm $N: K^*/K^{*3} \to F^*/F^{*3}$.
\end{proposition}

\begin{proof}
Since there is a unique rank 3 group scheme over~$F$ that becomes constant over $K_0$, it is enough to show that the kernel of the norm map $\Res^K_F  \mu_3 \longrightarrow \mu_3$ has this property. Over $\bar F$, this kernel is the subgroup $H = \{(\zeta^i, \zeta^{-i})\}$ of $\mu_3(\bar F) \times \mu_3(\bar F)$, where $\zeta$ is a primitive third root of unity. The two $\mu_3(\bar F)$ factors are indexed by the two embeddings $\tau_1$ and $\tau_2$ of $K$ into $\bar F$, and the action of  $g \in G_F$ is given by $g \cdot (a_j)_{\tau_j} = (g(a_j))_{g \circ \tau_j}$.  A simple computation then shows that $H$ is fixed by the subgroup $G_{K_0} \subset G_F$, since $K_0$ is the third quadratic algebra in the biquadratic extension $K\otimes_F F[\zeta]$ over~$F$.  For a more geometric proof of the proposition, see \cite[Prop.\ 18]{j=0}.    
\end{proof}

We combine Corollary \ref{bij} and Proposition \ref{cassels}:
\begin{theorem}\label{H1bij}
There is a natural bijection between $H^1(G_F, A[\phi])$ and the $\SL_2(F)$-orbits on $V(F)_{4\hat D}$.    Moroever, the stabilizer in $\SL_2(F)$ of any $f \in V(F)_{4\hat D}$ is isomorphic to $A[\phi](F)$. 
\end{theorem}

\subsection*{Elliptic curves}

Let us specialize to the case where $A = E$ and $A' = E'$ are elliptic curves.  An elliptic curve with a 3-isogeny admits a model of the form  
\begin{equation}\label{model}
E \colon y^2 = x^3 + D(ax+ b)^2
\end{equation}
with $a$, $b$, and $D$ in $F$.  The points $(0, \pm b\sqrt D)$ generate the kernel $E[\phi](\bar F)$.  In particular, the quadratic $F$-algebra $K_0 = F[t]/(t^2 - D)$ is independent of the chosen model of $E$, and is the same $K_0$ defined above.  The curve $E'$ admits a dual 3-isogeny $\hat\phi \colon E' \to E$, and a model for $E'$ is \cite[1.3]{cohenpazuki}
\begin{equation}\label{Eprime}
E' \colon y^2 = x^3 - 3D\left(ax + 3b - \frac49a^3D\right)^2.
\end{equation}

We phrase everything below in terms of $\hat\phi$ because it is cleaner.  In particular, the set $H^1(G_F, E'[\hat\phi])$ is in bijection with $\SL_2(F)$-equivalences classes of cubic forms of discriminant $4D$, and not $4\hat D$.  On the other hand, the group $H^1(G_F, E'[\hat\phi])$ also parameterizes isomorphism classes of $\hat\phi$-coverings, i.e.\ maps $C \to E$ of curves that become isomorphic to $\hat\phi$ over $\bar F$.  Because it may be useful in other applications, we give a clean description of these $\hat\phi$-coverings, using the bijection in Theorem \ref{H1bij}:      

\begin{proposition}\label{curve}
Let $C \to E$ be a $\hat\phi$-covering corresponding to $f(x,y) \in V(F)_{4D}$ under the bijection of Theorem $\ref{H1bij}$.  Then $C$ is isomorphic to the cubic curve
\[C \colon f(x,y) + ah(x,y)z + bz^3 = 0,\]
in $\P^2$.  Here, $h(x,y)$ is the scaled Hessian $h(x,y) = -\frac{1}{36}(f_{xx} f_{yy} - f_{xy}^2)$ of $f(x,y)$.        
\end{proposition}

\begin{proof}
Let $\delta \in K^*$ be a representative for the class corresponding to $f$.  From \cite[Thm.\ 4.1, Rem.\ (1)]{cohenpazuki}, we deduce the following equation for $C$ in $\P^2 = \P(K \oplus F)$ with coordinates $(w, z) \in K \oplus F$:
\[C_\delta \colon \frac{1}{2D}\Tr(\delta \tau w^3) + a\alpha zN(w) + bz^3 = 0,\]
where $\alpha$ is a cube root of $N(\delta)$ and $\tau$ is the image of $t$ in $K$.  Note that replacing $\delta$ by $\delta \kappa^3$, for any $\kappa \in K^*$, gives an isomorphic curve.  Thus, we may assume $f(x,y) = \frac{1}{2D}\Tr( \delta\tau (x + \tau y)^3)$, by the proof of Corollary \ref{bij}.  One computes that $\disc(f) = 4DN(\delta)$, so we in fact have $\alpha = 1$.  Another computation shows that the Hessian of $f(x,y)$ is $N(x + \tau y) = N(w)$, so the equation for $C_\delta$ has the desired form.    
\end{proof}

We can also describe the covering map $\pi: C \to E$.  Recall the syzygy \cite[Rem.\ 25]{j=0} 
\begin{equation}\label{syz}
(g/3)^2 =  \disc(f)f^2 + 4h^3
\end{equation}
satisfied by the covariants of $f$; here $g = f_xh_y - f_yh_x$ is the Jacobian derivative of $f$ and $h$.  Writing $f = -bz^3 -ahz$ and dividing (\ref{syz}) by $4z^6$, we obtain
\[\left(\dfrac{g}{6z^3}\right)^2 = \left(\dfrac{h}{z^2}\right)^3 + D\left(a \left(\dfrac{h}{z^2}\right) + b\right)^2.\]
This gives a map $\pi : C \to E$ sending $(z,y/x)$ to $(h/z^2, g/6z^3)$.

\section{$\phi$-soluble orbits of binary cubic forms}\label{soluble orbits}
As before, let $A \to A'$ be a 3-isogeny of abelian varieties over a field $F$.  There is then a bijection between the elements of $H^1(G_F, A[\phi])$ and the $\SL_2(F)$-orbits on the set $V(F)_{4\hat D}$ of binary forms of discriminant $4\hat D$, where $\hat D \in F^*$ is defined as in the previous section.  

As in the elliptic curve case, the group $H^1(G_F, A[\phi])$ classifies $\phi$-coverings $T \to A$, where $T$ is a torsor for $A'$. The orbits in the image of the Kummer map
\[\partial \colon A'(F) \to H^1(G_F, A[\phi])\]
correspond to {\it soluble} $\phi$-coverings $T \to A'$, i.e., $\phi$-coverings with $T(F) \ne \emptyset$.  We let $V(F)_{4D}^\sol$ denote the set of binary cubic forms $f(x,y)\in V(F)_{4D}$ corresponding, under the bijection of Theorem \ref{H1bij}, to a class in the image of the Kummer map.  If $ f\in V(F)^\sol_{4D}$, we say that $f$ is {\it $\phi$-soluble}.  

For each $s \in F^*$, we have the quadratic twists $A_s$ and $A'_s$ and a 3-isogeny $\phi_s \colon A_s \to A_s'$.  If $A = E$ is an elliptic curve with model $E \colon y^2 = x^3 + D(ax + b)^2$ then a model for the quadratic twist $E_s$ is 
\[E_s : y^2 = x^3 + Ds(ax+bs)^2,\]
and $\phi_s$ is the map given by taking the quotient by the points $(0,\pm bs\sqrt{Ds})$.  
For higher dimensional abelian varieties, one defines $A_s$, $A'_s$, and $\phi_s$ via descent, using the cocycle $\chi_s \colon G_F \to \{\pm1\}$ corresponding to the extension $F(\sqrt{s})/F$; see \cite[III.1.3]{Serre} for example.  Note that it makes sense to twist the 3-isogeny $\phi$ by $\chi_s$ since $A[\phi]$ is preserved by the inversion automorphism on $A$.  We often omit the subscript from $\phi_s$ when the context makes it clear---for example, when writing $A_s[\phi]$.  

The key fact we will use about $\phi_s$ is that 
\[A_s[\phi_s](\bar F) \simeq A[\phi](\bar F) \otimes \chi_s\] 
as $\F_3[G_F]$-modules.  In particular, if $F[t]/(t^2 - \hat D)$ is the mirror algebra of $A[\phi]$, then $F[t]/(t^2 - \hat Ds)$ is the mirror algebra for $A_s[\phi_s]$.  This can be read directly off the models in the elliptic curve case.

By Theorem \ref{H1bij} applied to $\phi_s$, the set of $\SL_2(F)$-orbits on the set $V(F)_{4\hat Ds}$ of binary cubic forms over~$F$ of discriminant $4\hat Ds$ is in bijection with $H^1(G_F, A_s[\phi])$.  This allows us to define a notion of $\phi$-solubility for binary cubic forms over~$F$ of any discriminant, compatible with the already defined notion of $\phi$-solubility for forms of discriminant $4\hat D$:

\begin{definition}{\em
A binary cubic form $f \in V(F)$ is {\it $\phi$-soluble} if it corresponds under the bijection of Theorem \ref{H1bij} to an element in the image of $\partial\colon A'_s(F) \to  H^1(G_F, A_s[\phi])$, where $s = \disc(f)/(4\hat D)$.  The set of all $\phi$-soluble $f \in V(F)$ is denoted by $V(F)^{\sol}$. 
}\end{definition}

\begin{remark}{\em
For any $s \in F^*$, we have $V(F)^{\sol} = V(F)^{\sols}$.  In other words, this notion of $\phi$-solubility depends only on the quadratic twist family of $A$ and the associated family of 3-isogenies, $\{\phi_s:A_s\to A'_s\}$, and not on the isogeny $\phi:A\to A'$ itself.} 
\end{remark}

Our main theorem parametrizing elements of $A'_s(F)/\phi(A_s(F))$, $s\in F^*$, by $\phi$-soluble binary cubic forms over~$F$ is as follows. 

\begin{theorem}\label{soluble}
There is a natural bijection between  the $\SL_2(F)$-orbits on
  $V(F)^\sol$ having discriminant~$4\hat Ds$ and the elements of the group $A'_s(F)/\phi(A_s(F))$.  Under this bijection, the identity element of $A'_s(F)/\phi(A_s(F))$ corresponds to the unique $\SL_2(F)$-orbit of reducible binary cubic forms over~$F$ of
  discriminant $4\hat Ds$, namely the orbit of $f(x,y) = 3x^2y + \hat Dsy^3.$    
\end{theorem}
 \begin{proof}
The image of the Kummer map for $A_s$ is isomorphic to $A'_s(F)/\phi(A_s(F))$, so the first statement follows. For the second part of the corollary, we use the explicit description of the bijection given in (the proof of) Corollary \ref{bij}.  If $\delta \in \left(K^*/K^{*3}\right)_{N= 1}$, the corresponding cubic form is $\frac{1}{2D} \Tr_{K/F}(\delta \tau X^3)$.  This is reducible over~$F$ if and only if $\delta \in F^\times$ (modulo cubes in $K$), since $\tau$ is the unique element in $K$ of trace-zero up to $F$-scaling.  If this is the case, then $N(\delta) = \delta^2 \in F^{*3}$, so we must have $\delta \in F^{*3}$, i.e.\ $\delta$ is the class of the identity in $\left(K^*/K^{*3}\right)_{N= 1}$.    
\end{proof}

\section{$\phi$-soluble orbits over local fields}

Now assume that $p$ is a prime and $F$ is a finite extension of $\Q_p$ with ring of integers $\O_F$.  Fix a 3-isogeny $\phi \colon A \to A'$ of abelian varieties over~$F$.  The 3-isogeny $\phi$ determines a notion of $\phi$-solubility, as defined in the previous section, on the space $V(F)$ of binary cubic forms. 

In this section we show that, in certain circumstances, $\phi$-solubility of $f \in V(F)$ implies that $f$ is $\SL_2(F)$-equivalent to an integral form $\tilde f \in V(\O_F)$.  We find that the conditions under which we can guarantee integrality 
are related to the value of the Selmer ratio
\[c(\phi) :=  \dfrac{|\coker\,  \phi \colon A(F) \to A'(F)|}{|\ker \phi \colon A(F) \to A'(F)|}\]
defined in the introduction.

As before, let $K = F[t]/(t^2 - \hat D)$ be the mirror quadratic $F$-algebra attached to $A[\phi]$.  Assume from now on that $\hat D$ lies in $\O_F$.  We first consider the case of good reduction. 

\begin{proposition}\label{solsgood}
If $p \neq 3$ and $A/F$ has good reduction, then the Kummer map induces an isomorphism
\[A'(F)/\phi(A(F)) \simeq (\O_K^*/\O_K^{*3})_{N = 1},\]
where $\O_K$ is the ring of integers of $K$ and 
\[(\O_K^*/\O_K^{*3})_{N = 1} = \ker \left(N \colon \O_K^*/\O_K^{*3} \to \O_F^*/\O_F^{*3}\right).\]  
Moreover, $c(\phi) = 1$ in this case.
\end{proposition}

\begin{proof}
Under the hypotheses, the image of the Kummer map 
\[A'(F) \to H^1(G_F,A[\phi]) \]
is exactly the subgroup $H^1_\mathrm{un}(G_F, A[\phi])$ of unramified classes, i.e., those that become trivial when restricted to the inertia subgroup of $G_F$.  In the case of elliptic curves, this is \cite[Lem.\ 4.1]{Cassels8}.  For the general case, see \cite[Prop.\ 2.7(d)]{Cesnavicius2}.  On the other hand, by local class field theory, the unramified classes map onto $\left(\O_K^*/\O_K^{*3}\right)_{N = 1}$ under the isomorphism of Proposition \ref{cassels}, which proves the first claim.

For the second claim, first note that $\hat D$ may be taken to be a unit in $\O_F$ in the good reduction case.  This follows from the criterion of Ogg-N\'eron-Shaferevich.  For an elementary proof in the elliptic curve case, note that $D$ divides the discriminant of the elliptic curve (\ref{model}).  It follows that $\O_K \simeq \O_F[t]/(t^2 - \hat D)$.  We then compute \[|A'(F)/\phi(A(F))| = |(\O_K^*/\O_K^{*3})_{N = 1}|  = |\O_K^*[3]_{N = 1}| = |A[\phi](F)|,\]
where we have used Theorem \ref{cubicbij} and Theorem \ref{H1bij}.  This shows that $c(\phi) = 1$.  
\end{proof}

\begin{theorem}\label{integrality}
Suppose that $p >3$, that $A$ has a quadratic twist of good reduction, and that $\hat D$ is a not in the square of the maximal ideal in $\O_F$.  Then any $\phi$-soluble $f \in V(F)_{4\hat D}$ is $\SL_2(F)$-equivalent to an integral form $\tilde f \in V(\O_F)_{4\hat D}$.  Moreover, we have $c(\phi) = 1$.        
\end{theorem}

\begin{proof}
First, suppose that $A$ has good reduction.  By the Ogg-N\'eron-Shaferevich criterion, the kernel $A[\phi]$ is unramified, so that $K$ is unramified over~$F$ (since $p \neq 3$).  Since $p \neq 2$, it follows that $\hat D$ is a unit in $\O_F$, and  $\O_F[t]/(t^2 - \hat D)$ is the maximal order in $K$.  By Proposition \ref{intorb}, the orbit of $f$ corresponds under the bijection of Theorem \ref{soluble} to a class in $(\O_K^*/\O_K^{*3})_{N = 1}$.  Thus by Proposition \ref{solsgood}, the orbit of $f$ contains an integral orbit, and $c(\phi) = 1$.

Next, assume that $A$ has bad reduction.  By assumption, it must be a {\it ramified} twist of an abelian variety $B$ with good reduction, i.e., $A = B_s$ for some $s$ in the maximal ideal of~$\O_F$.  It follows that the mirror algebra $K$ is ramified over~$F$.  But then the group $\left(K^*/K^{*3}\right)_{N = 1}$ is trivial, and $f$ is $\SL_2(F)$-equivalent to the integral form $3x^2y + \hat Dsy^3$.  

Finally, we need to show that $c(\phi) = 1$ in this bad reduction case.  The numerator in the definition of $c(\phi)$ equals 1 since 
\[|A'(F)/\phi(A(F))| \leq \left|\left(K^*/K^{*3}\right)_{N = 1}\right| = 1.\] 
On the other hand, the denominator equals $|A[\phi](F)|$, which also equals 1.   Indeed, the field $K_0 = F[t]/(t^2 +3\hat D)$, over which the Galois action on $A[\phi]$ trivializes, is ramified over~$F$ (since the mirror $F$-algebra $K$ is ramified, and since $p \neq 3$).  Thus $A[\phi](F)$ is trivial.   
\end{proof}

\section{$\phi$-Selmer groups and locally $\phi$-soluble orbits over a global field}\label{integerorbits}

Now let $F$ be a number field.  Write $F_v$ for the completion of $F$ at a place $v$.  If $\varphi : A \to A'$ is an isogeny of abelian varieties over~$F$, the {\it $\varphi$-Selmer group} $\Sel_\varphi(A)$ is the subgroup of  $H^1(G_F, A[\varphi])$ of classes that are locally in the image of the Kummer map 
\[\partial_v : A'(F_v) \la H^1(G_{F_v}, A[\varphi])\]
for every place $v$ of $F$.  Equivalently, these are the classes locally in the kernel of the map   
\[H^1(G_{F_v}, A[\varphi]) \to H^1(G_{F_v}, A)\]
for every place $v$, i.e.\ the classes corresponding to principal homogeneous spaces with an $F_v$-point for every place $v$. 

Now assume $\phi \colon A \to A'$ is a 3-isogeny over~$F$.  If $v$ is a place of $F$, write $A_v$ for the base change $A \otimes_F F_v$.
We have defined the subset $V(F)^{\sol} \subset V(F)$ of $\phi$-soluble cubic forms over~$F$.  For each $v$, we have also defined the subset $V(F_v)^{\sol} \subset V(F_v)$ of $\phi$-soluble cubic forms over $F_v$.   We let $V(F)^\ls \subset V(F)^{\sol}$ denote the set of {\it locally $\phi$-soluble} binary cubic forms, i.e.\ the set of $f \in V(F)$ such that $f \in V(F_v)^{\sol}$ for all places $v$ of $F$.  

Finally, fix $\hat D \in \O_F$ so that $K = F[t]/(t^2 - \hat D)$ is the mirror algebra associated to $A[\phi]$.  The following theorem now follows immediately from Theorem \ref{soluble}.  

\begin{theorem}\label{Qcor}
Let $\phi \colon A \to A'$ be a $3$-isogeny of abelian varieties over a number field $F$, and fix $s \in F^*$.  Then there is a natural bijection between the $\SL_2(F)$-orbits of
  $V(F)^\ls$ of discriminant $4\hat Ds$ and the elements of the
  $\phi$-Selmer group $\Sel_{\phi}(A_s)$ corresponding to the isogeny $\phi_s:A_s\to A'_s.$  Under this bijection, the identity element of $\Sel_{\phi}(A_s)$ corresponds to the unique orbit of reducible binary cubic forms of discriminant $4\hat D s$.   
\end{theorem}

Theorem \ref{Qcor} will be even more useful to us if we can find, for each locally $\phi$-soluble $f \in V(F)_{4\hat Ds}$, an $\SL_2(F)$-equivalent {\it integral} form $\tilde f \in V(\O_F)_{4\hat Ds}$.  This is clearly only possible if the product $4\hat Ds$ is itself integral.  But even if $s$ lies in $\O_F$, it is not generally the case that a locally $\phi$-soluble $\SL_2(F)$-orbit contains any integral cubic forms.

To circumvent this issue, we can choose $N \in \O_F$ such that $f$ is $\GL_2(F)$-equivalent to an integral form of discriminant $4\hat DsN^2$, under the twisted $\GL_2(F)$-action of Remark \ref{twisted}.    For example, we could act by a sufficiently divisible scalar matrix.  Note that ``twisting'' the discriminant by $N^2$ just amounts to replacing $\hat D$ by $\hat D N^2$.  The next theorem shows that we can choose the twisting factor~$N$ independent of $f$ and, indeed,  independent of $s$.

\begin{theorem}\label{uniform integrality}
Let $\phi \colon A \to A'$ be a $3$-isogeny over a number field $F$.  Then we can choose the parameter $\hat D \in \O_F$ such that for all finite places $v$ and for all $s \in \O_{F_v}$ not in the square of the maximal ideal, the following property holds: 
\begin{enumerate}[$(*)$]
\item 
If $f \in V(F_v)_{4\hat D s}^\sol$, then $f$ is $\SL_2(F_v)$-equivalent to a binary cubic form $\tilde f \in V(\O_{F_v})$.     
\end{enumerate}
\end{theorem}

\begin{proof}
To guarantee property $(*)$ at a finite place $v$, it is enough to verify it on a set of representatives $\{s_1,\ldots, s_k\}$ of $F^*/F^{*2}$ that are not in the square of the maximal ideal of~$\O_{F_v}$.  For each $s_i$, we choose a finite set $\{f_{i1}, \ldots, f_{it}\}$ of representatives for the $\SL_2(F_v)$-orbits of $\phi$-soluble cubic forms of discriminant $4s_i$.  For each $f_{ij}$ we can clear the denominators of $f_{ij}$ by twisting by some appropriate integer $N$.  There is therefore a single integer $N$ that clears the denominators of all these $f_{ij}$.   Thus for any finite set of places $S$, we can find $\hat D \in \O_F$ such that property $(*)$ holds at all $v \in S$.  

Choose such a $\hat D \in \O_F$ for a set $S$ containing all places of bad reduction for $A$ as well as all places above 2 and 3.  By enlarging $S$, we may even assume that $\ord_v(\hat D) < 2$ for all finite primes $v$ not in $S$.  Indeed, for the finitely many finite primes $v$ such that $\ord_v(\hat D) \geq 2$, we may scale $\hat D$ by a power of the generator of the principal ideal $\p_v^h$, where $h$ is the class number of $F$, so that $(*)$ holds at $v$.  It remain to show that $(*)$ holds for all finite places $v \notin S$, but this follows from Theorem \ref{integrality}.  
\end{proof}

In the case $F = \Q$, we have the following global result.  

\begin{theorem}\label{solint}
Let $\Sigma \subset \Z$ be a subset of squarefree integers defined by finitely many congruence conditions.  Then there exists an integer $\hat D$, depending only on $\Sigma$ and $A$, with the following property:  if $s \in \Sigma$, then every locally $\phi$-soluble $\SL_2(\Q)$-orbit on $V(\Q)$ of discriminant $4\hat Ds$ contains an integral representative $f \in V(\Z)$.  
\end{theorem}

\begin{proof}
This follows from Theorem \ref{uniform integrality} and the fact that $\SL_2$ has class number 1.  Alternatively, this follows from Theorems \ref{integrality} and \ref{cubicbij}, together with the fact that $\Z$ is a PID.   
\end{proof}

\begin{corollary}\label{integralparam}
Let $A$, $\Sigma$, and $\hat D$ be as in Theorem $\ref{solint}$.  Then for any $s \in \Sigma$, there is a bijection between $\Sel_{\phi}(A_s)$ and the $\SL_2(\Q)$-orbits of locally $\phi$-soluble $f \in V(\Z)$ of discriminant $4\hat Ds$.  
\end{corollary}

\begin{remark}{\em 
A similar result holds for any number field $F$ of class number 1, but we will not use this fact explicitly in our proof of Theorem \ref{main}.}
\end{remark}

\section{The average size of $\Sel_\phi(A_s)$}\label{counting}
In this section we prove Theorem \ref{main}, by counting equivalence classes of everywhere locally $\phi$-soluble integral binary cubic forms of bounded discriminant.  If $F = \Q$, one can give a proof similar in flavor to the proof in \cite{j=0}, using Theorem \ref{solint}.  Since we wish to handle quadratic twist families over an arbitrary number field $F$, we will use the general counting results of Bhargava, Shankar, and Wang \cite{BSW2} over global fields.  We recall the setup and notation.

Let $\phi \colon A \to A'$ be a 3-isogeny of abelian varieties over~$F$.  For each $s \in F^*/F^{*2}$, we have the quadratic twist $\phi_s: A_s \to A'_s$.  Consider the action of $F^*$ on $F$ given by $\alpha.s = \alpha^2 s$.  The height function on $F^*/F^{*2}$ from the introduction lifts to a function on $F$: if $\tilde s$ is a representative of $s \in F^*/F^{*2}$, and $I(\tilde s)$ is the ideal 
\[I(\tilde s) = \{a \in F \colon a^2\tilde s \in \O_F \},\]
then 
\[H(s) = (NI(\tilde s))^2\prod_{\p \in M_\infty} |\tilde s|_\p,\]
where $M_\infty$ denotes the set of infinite places of $F$.  There is also a natural height function on the set $F_\infty := \prod_{\p \in M_\infty} F_\p$, defined by 
\[H((s_\p)_{\p \in M_\infty}) = \prod_{\p \in M_\infty} |s_\p|_\p.\]

In order to take averages over subsets of twists defined by local conditions on $s$, we require a notion of functions on $F$ that are defined by local conditions.  We say a function $\psi \colon F \to [0,1]$ is {\it defined by local congruence conditions} if there exist local functions $\psi_\p \colon F_\p \to [0,1]$ for every finite place $\p$ of $F$, and a function $\psi_\infty \colon F_\infty \to [0,1]$, such that the following two conditions hold:
\begin{enumerate}
\item For all $w \in F$, the product $\psi_\infty(w)\prod_{\p \notin M_\infty} \psi_\p(w)$ converges to $\psi(w)$.
\item For each finite place $\p$, and for $\p = \infty$, the function $\psi_\p$ is nonzero on some open set and locally constant outside some closed subset of $F_\p$ of measure 0.  
\end{enumerate}
A subset of $F$ is said to be {\it defined by local congruence conditions} if its characteristic function is defined by local congruence conditions.   

Let $\Sigma_0$ be the fundamental domain for the above action of $F^*$ on $F$, constructed in \cite[\S3.4]{BSW2}.  Then $\Sigma_0$ is defined by local congruence conditions.  For any $X > 0$, we denote $F_X$ denote the set of $s \in F$ with $H(s) < X$.  Then $\Sigma_0 \cap F_X$ is finite.  We will think of the abelian varieties $A_s$ as elements of $\Sigma_0$.

A family of quadratic twists defined by congruence conditions is a subset $\Sigma_1 \subset \Sigma_0$ defined by local congruence conditions.  In that case, the characteristic function  $\chi_{\Sigma_1}$ of $\Sigma_1$ factors as 
\[\chi_{\Sigma_1, \infty} \prod_{\p \notin M_\infty} \chi_{\Sigma_1, \p}.\]  
For each finite place $\p$ of $F$, let $\Sigma_{1,\p}$ be the subset of $F_\p$ whose characteristic function is $\chi_{\Sigma_1, p}$, and let $\Sigma_{1,\infty}$ be the subset of $F_\infty$ whose characteristic function is $\chi_{\Sigma_1, \infty}$.  Let $\O_\p$ denote the completion of the ring of integers $\O_F$ at $\p$ and $v_\p$ the $\p$-adic valuation normalized so that the valuation of a uniformizer is 1.   A family of quadratic twists given by $\Sigma_1$ is {\it large} if $\Sigma_{1,\p}$ is the set $\O_\p(2) = \{s \in \O_\p \colon v_\p(s) < 2\}$, for all but finitely many finite place $\p$, and if $\Sigma_{1,\infty}$ is a non-empty union of cosets in $F^*_\infty/F_\infty^{*2}$.  We recall that $\Sigma_{0,\p} = \O_\p(2)$ for all finite $\p$, so $\Sigma_0$ is itself large.

We define 
\begin{equation}\label{cinftyphis}
c_\infty(\phi_s) = \prod_{\p \in M_\infty} c_\p(\phi_{s_\p}),
\end{equation}
for any $s  = (s_\p) \in F^*_\infty$.  We let $F_{\infty, X}$ be the set of $s \in F_\infty$ with height less than $X$.              
  
\begin{theorem}\label{globalS}Let $\Sigma_1$ be a large family of quadratic twists of $E$.  When the elliptic curves $E_s$, $s \in \Sigma_1$, are ordered by height of $H(s)$ of $s$, the average size of $\Sel_{\phi_s}(E_s)$ is
\[1 + \dfrac{\displaystyle\int_{s \in \Sigma_{1,\infty} \cap F_{\infty,1}} c_\infty(\phi_s) \,  d\mu_\infty^*(s)}{\displaystyle\int_{s \in \Sigma_{1,\infty} \cap F_{\infty,1}} d\mu_\infty^*(s)} \cdot \prod_{\p \notin M_\infty} \dfrac{\displaystyle\int_{s \in \Sigma_{1,\p}} c_\p(\phi_s) ds}{\displaystyle\int_{s \in \Sigma_{1,\p}} ds}.\]
Here $ds$ denotes a Haar measure on $\O_\p$, 
and $d\mu_\infty^*$ a Haar measure on $F_\infty$.
\end{theorem}   

By Theorem \ref{integrality}, the Euler product above is in fact a finite product, and it is exactly equal to the average size of $c(\phi_s)$ for $s \in F^*/F^{*2}$.  
\begin{proof}
By Corollary \ref{Qcor}, it suffices to count the number of irreducible locally $\phi$-soluble $\SL_2(F)$-orbits on $V(F)$ with discriminant in $\Sigma_1$ and with height less than $X$.   Theorem~\ref{globalS} then follows from the very general counting result  \cite[Thm.~13]{BSW2}.  To apply that result to our situation,  we take $G = \SL_2$ and $V = \Sym^3\,  2$, the space of binary cubic forms.  The GIT quotient $S$ is the affine line $\A^1$ and the map $\inv\colon V \to S$ is the discriminant $f \mapsto \Disc(f)$.   We take $V(F)^\irr$ to be the subset of irreducible cubic forms, and the weight function $m_0$ is the characteristic function of $V(F)^{\ls, E} \cap \inv^{-1}(\kappa.\Sigma_1) \subset V(F)$, for some non-zero $\kappa \in \O_F$, which we will choose momentarily.  That the $\p$-adic integrals in Theorem \ref{globalS} coincide with those in \cite[Thm.~13]{BSW2} follows from Theorem~\ref{H1bij} and Theorem~\ref{soluble}.  Finally, we note that the Tamagawa number $\tau_{\SL_2,F}$ is equal to 1.  Thus, all that remains is to verify the six axioms in \cite[Thm.\ 13]{BSW2} for these choices of input.        

Axiom - (G,V) is satisfied since $\Delta$ is a degree 4 polynomial in the coordinates of $V$ and since $\SL_2$ is semisimple.  To guarantee Axiom - Local Condition is satisfied, we must choose $\kappa$ carefully, but Theorem \ref{integrality} exactly says that such a $\kappa$ exists.  Axiom - Local Spreading is satisfied since there is a section of $\inv \colon V \to S$ given by $s \mapsto 3x^2y + (s/4)y^3$, which is defined over $\O_F[1/2]$.  One verifies Axiom - Counting at Infinity I and II exactly as in the case \cite[\S4.1]{BSW2} of $\PGL_2$ acting on binary quartic forms.  Finally, to verify Axiom - Uniformity Estimate in our situation, note that by Propositions~\ref{intorb} and \ref{solsgood}, the set $\inv^{-1}(\Sigma_{1,\p}) \subset V(\O_\p)$ contains all cubic forms with discriminant not divisible by $\p^2$, for all but finitely many primes $\p$ of $F$.  The uniformity estimate therefore follows from \cite[Thm.~17]{BSW1}.\end{proof}  

\section{Applications to ranks of abelian varieties}\label{sec:abvarranks}
In this section we prove Theorems \ref{ab var ranks}, \ref{avgrank}, and \ref{globalbounds}.  

\begin{lemma}\label{lem:ranks-add}
Let $\phi_1\colon A_1 \to A_2$ and $\phi_2 \colon A_2 \to A_3$ be isogenies of abelian varieties and set $\psi = \phi_2 \circ \phi_1$.  Then there is an exact sequence
\[
\mathrm{Sel}_{\phi_1}(A_1) \to \mathrm{Sel}_\psi(A_1) \to \mathrm{Sel}_{\phi_2}(A_2).
\]
\end{lemma}
\begin{proof}
In fact, a standard diagram chase yields the exact sequence
\begin{equation}\label{fiveterm}
0 \to A_2(F)[\phi_2] /\phi_1(A_1(F)[\psi]) \to \mathrm{Sel}_{\phi_1}(A_1) \to \mathrm{Sel}_\psi(A_1) \to \mathrm{Sel}_{\phi_2}(A_2) \to \dfrac{\Sha(A_2)[\phi_2]}{\phi_1(\Sha(A_1)[\psi])} \to 0. 
\end{equation}
\end{proof}

\begin{proof}[Proof of Theorem $\ref{ab var ranks}$]
Suppose $\psi \colon A \to A$ is an isogeny which factors as $\phi_1 \circ \dots \circ \phi_n$, with each $\phi_i$ of degree $3$. Let $A_i$ denote the domain of $\phi_i$, and let $\O$ be the subring of $\End(A) \simeq \End(A_s)$ generated by $\psi$.  Let $r(A_s)$ be the rank of $A_s(F)$.  Write $t$ for the rank of $\O$ as a $\Z$-module, and suppose the ring $\O/\psi$ has size $3^e$.  Then the group $A_s(F)/\psi A_s(F)$ has size at least $3^{\frac{e}{t}r(A_s)}$ and embeds in $\Sel_\psi(A_s)$.  Each of the groups $\Sel_{\phi_i}(A_{i,s})$ are $\F_3$-vector spaces, so by Lemma \ref{lem:ranks-add} and the trivial inequality $k \leq 3^k$, for every $k\in \mathbb{Z}$, we find 
\[
\frac{e}{t}r(A_s) \leq \sum_{i=1}^{n} \dim_{\F_3}\mathrm{Sel}_{\phi_{i,s}} (A_{i,s}) \leq \sum_{i=1}^{n} |\mathrm{Sel}_{\phi_{i,s}}(A_{i,s})|,
\]
for any $s\in F^*/F^{*2}$.  Upon taking the average over $s$, Theorem \ref{main} immediately yields the result.
\end{proof}

\subsection{Proof of Theorem \ref{avgrank}}
We take $A = E$ in Theorem \ref{main} and recall that we wish to establish an explicit upper bound on the average rank of the twists $E_s$.  For each $m \in \Z$, we define $T_m(\phi) = \{s \in F^*/F^{*2} \colon c(\phi_s) = 3^m\}$.  From the proof of Theorem \ref{main}, we see that $T_m$ is either empty or cut out by finitely many local conditions.  Moreover, the average size of $\Sel_\phi(E_s)$ for $s \in T_m$ is $1 + 3^m$.  Cassels' formula \cite{Cassels8},
\begin{equation}\label{cassels formula}
c(\phi) = \dfrac{|\Sel_\phi(E)||E'[\hat\phi](\Q)|}{|\Sel_{\hat\phi}(E')||E[\phi](\Q)|},
\end{equation} 
shows that $c(\phi)c(\hat\phi) = 1$.    Thus the average size of $\Sel_{\hat \phi}(E'_s)$ for $s \in T_m$ is $1 + 3^{-m}$.  The proof of Theorem \ref{avgrank} is now exactly as in \cite[Thm.\ 43]{j=0}.  Namely, a convexity bound shows that the average {rank} of $\Sel_\phi(E_s) \oplus \Sel_{\hat \phi}(E'_s)$ for $s \in T_m$ is at most $|m| + 3^{-|m|}$.  Since this rank is an upper bound on the rank of $\Sel_3(E_s)$, which is itself an upper bound on the rank of $E_s(F)$, the theorem follows for any $\Sigma$ contained in $T_m$.  The general case follows by writing $\Sigma = \bigcup_m \left(\Sigma \cap T_m\right)$ and adding together the rank bounds on each $\Sigma \cap T_m$ weighted by their density.       

\subsection{Proof of Theorem \ref{globalbounds}}

We now adapt the ideas of the proof of Theorem \ref{avgrank} to provide an explicit proportion of twists $E_s$ that have rank 0 or $3$-Selmer rank 1.  Let $T_m(\phi)$ be as in the previous proof, so that the average of the rank of $\mathrm{Sel}_3(E_s)$ for $s\in T_0(\phi)$ is at most $1$.  The only additional input needed is a result of Cassels \cite{Cassels8} which shows that if $s \in T_m(\phi)$, then $\dim_{\mathbb{F}_3} \mathrm{Sel}_3(E_s) \equiv m \pmod{2}$; see \cite[Prop.\ 42(ii)]{j=0}.  In particular, every twist within $T_0(\phi)$ has even $3$-Selmer rank, so it follows that at least $50\%$ of the twists in $T_0(\phi)$ have rank 0.  Similarly, for twists in either $T_1(\phi)$ or $T_{-1}(\phi)$, the $3$-Selmer rank is odd and the average $3$-Selmer rank is $4/3$.  It then immediately follows that at least $5/6$ of such twists must have $3$-Selmer rank $1$.

\section{Computing local $\phi$-Selmer ratios}\label{local computations}
Suppose $F$ is a finite extension of $\Q_p$ or $\R$, and let $\phi \colon A \to A'$ be a 3-isogeny of abelian varieties over~$F$.  The goal of this section is to give explicit formulas for the (local) Selmer ratio  \[c(\phi) := \dfrac{|\coker\,  \phi \colon A(F) \to A'(F)|}{|\ker \phi \colon A(F) \to A'(F)|}.\]

Since we have fixed $F$ to be a local field in this section, we drop the subscript $\p$ from the notation $c_\p$ used in the introduction.  There should be no confusion with the {\it global} Selmer ratio which is only defined for isogenies over global fields, and which will not be discussed in this section.   

\subsection{Selmer ratios over $\R$ and $\C$}
Computing $c(\phi)$ is simple when $F$ is archimedean.  If $F  = \R$ or $\C$, then every binary cubic form $f \in V(F)$ is reducible, and there is a unique $\SL_2(F)$-orbit of binary cubic forms of any fixed discriminant.  This corresponds, under the bijection in Theorem \ref{soluble}, to the fact that $A'(F)/\phi(A(F))$ is always trivial, since the component group of $A(F)$ is a 2-group.  We conclude:

\begin{proposition}\label{R}
If $F = \C$, then $c(\phi)  = 1/3$.  If $F = \R$, then 
\begin{equation*}c(\phi) = 
\begin{cases}
\frac{1}{3} & A[\phi](\R) \simeq \Z/3\Z; \\
1 & A[\phi](\R) = 0.
\end{cases}\end{equation*}
\end{proposition} 

\subsection{Selmer ratios over $p$-adic fields}
Let $p$ be a prime and suppose $F$ is a finite extension of $\Q_p$.  Let $\O_F$ be the ring of integers, $\pi$ a uniformizer, and $q$ the size of the residue field $\O_F/\pi$.    We write $v$ for the valuation on $F$, normalized so that $v(\pi) = 1$, and $| \cdot |$ for the absolute value on $F$, normalized so that $|\pi | = q^{-1}$.  
Our goal is to compute the value of $c(\phi)$ in terms of local invariants of $A$ and $A'$, and to study the behavior of $c(\phi)$ under twisting.  

For general abelian varieties (especially those of dimension greater than 1), the following formula is essentially all we can currently say about $c(\phi)$.

\begin{proposition}\label{lem:tamag} 
Suppose $F$ is a finite extension of $\Q_p$ for some $p$.  Then 
\[
c(\phi) = \dfrac{c(A^\prime)}{c(A)} \cdot |\phi'(0)|^{-1}
\]
where
$c(A) = A(F)/A_0(F)$ is the local Tamagawa number and $|\phi'(0)|$ is the normalized absolute value of the determinant of the Jacobian matrix of partial derivatives of $\phi$ evaluated at the origin. 
\end{proposition}
\begin{proof}
This is \cite[Lem.\ 3.8]{Schaefer}.
\end{proof}

\begin{corollary}
\label{cor:tamag}
If $p \ne 3$, then $c(\phi) = c(A^\prime)/c(A)$.  
\end{corollary}
\begin{proof}
We have $|\phi'(0)| = 1$ in this case; see \cite[pp.\ 92]{Schaefer}.
\end{proof}

If $A = E$ is an elliptic curve, then Proposition \ref{lem:tamag} leads to the following computation of $c(\phi)$ in the case $p \neq 3$.

\begin{proposition}\label{localratio}
Suppose $p \neq 3$ and $\phi \colon E \to E'$ is a $3$-isogeny of elliptic curves.  Then $c(\phi) = 1$ unless either
\begin{enumerate}[$($a$)$]
\item $E$ has split multiplicative reduction, or 
\item $F$ does not contain a primitive third root of unity and $E$ has reduction type $\IV$ or $\IVS$.   
\end{enumerate}
In case $(a)$, we have $c(\phi) = v(j(E'))/v(j(E))$, where $j$ is the $j$-invariant.

\noindent 
In case $(b)$, we have $c(\phi) = 3$ if $E[3](F) = 0$, and $c(\phi) = 1/3$ otherwise.  
\end{proposition}
\begin{proof}
This follows from \cite[Table 1]{DD}.  
\end{proof}

The computation of $c(\phi)$ in the case $p = 3$ is more involved because $c(E')/c(E)$ is difficult to analyze and we also  need to take into account the term $|\phi'(0)|^{-1}$ in Proposition \ref{lem:tamag}.  Since $A = E$ is now an elliptic curve, we may write $|\phi'(0)| = |\frac{\phi^* \omega^\prime}{\omega}|$, where $\omega$ and $\omega^\prime$ are N\'eron differentials on minimal models for $E$ and $E'$, respectively.  This ratio of differentials is difficult to describe purely in terms of the reduction type and other invariants of the minimal model.  The case of (potentially) supersingular reduction is especially difficult, and for those curves we still have an incomplete picture.   The next theorem collects what information we do know about these 3-adic invariants. We set $\alpha_{\phi,F} := |\phi'(0)|^{-1}$.  

\begin{theorem}\label{3adic}
Let $F$ be a finite extension of $\Q_3$ and $\phi \colon E \to E'$ a $3$-isogeny of elliptic curves over~$F$.  Then $c(\phi) = \frac{c(E')}{c(E)} \cdot \alpha_{\phi, F}$, and both factors are integer powers of $3$ satisfying $\frac 13 \leq c(E')/c(E) \leq 3$ and $1 \leq \alpha_{\phi, F} \leq 3^{[F \colon \Q_p]}$.  

The values of $c(E')/c(E)$ and $\alpha_{\phi,F}$ are as in Tables $\ref{tab:value of c}$ and $\ref{tab:values of alpha}$.  Here, we write $j$ and $j'$ for $j(E)$ and $j(E')$, $\Delta_\minn{}$ and $\Delta'_\minn{}$ for the minimal discriminants of $E$ and $E'$, and $d = [F\colon \Q_p]$.  We let $K = F(\sqrt{D})$ be the field over which the points of $E[\phi]$ are defined, and $\m_K$ the maximal ideal in~$\O_K$.  We also write $\hat E$ for the formal group of $E$.      
\end{theorem}

\begin{table}

\begin{center}
\begin{tabular}{@{\vrule width 1.2pt\ }l |@{\vrule width 1.2pt}  c |@{\vrule width 1.2pt}}
\noalign{\hrule height 1.2pt}

\vphantom{$\displaystyle{\int^\int}$}
{Reduction type of $E$}&$c(E^\prime)/c(E)$\\ [4pt]

\noalign{\hrule height 1.2pt}

\vphantom{$\int^{\int^\int}$}%
(potentially) ordinary  &  $1$\\ [4pt]

\hline 

\vphantom{$\int^{a^{\int^a}}$} 
split multiplicative  & $v(j')/v(j)$\\ [4pt]

\hline

\vphantom{$\int^{a^{\int^a}}$} 
potentially or nonsplit multiplicative &    $1$ \\[4pt]

\hline
\vphantom{$\int^{a^{\int^a}}$}

(potentially) supersingular  & $\frac13, 1, \mbox{or } 3$    \\ [4pt]
  
\noalign{\hrule height 1.2pt}

\noalign{\hrule height 1.2pt\phantom{$a_a$}}
\end{tabular}
\end{center}
\caption{The value of $c(E^\prime)/c(E)$ when $F$ is a $3$-adic field.}
\label{tab:value of c}

\end{table}

\begin{table}

\begin{center}
\begin{tabular}{@{\vrule width 1.2pt\ }l |@{\vrule width 1.2pt}  c |@{\vrule width 1.2pt}}
\noalign{\hrule height 1.2pt}

\vphantom{$\displaystyle{\int^\int}$}
{Reduction type of $E$}&$\alpha_{\phi, F}$\\ [4pt]

\noalign{\hrule height 1.2pt}

\vphantom{$\int^{\int^\int}$}%
(potentially) ordinary  & \\ [4pt]
  $\quad$ $E[\phi](K) \not\subset \hat E(\m_K)$ & $1$   \\ [4pt] 
  $\quad$ $E[\phi](K) \subset \hat E(\m_K)$ & $3^d$ \\ [4pt]

\hline 

\vphantom{$\int^{a^{\int^a}}$} 
(potentially) multiplicative  & \\ [4pt]
  $\quad$ $v(j)\!=\!3v(j')$ & $1$   \\ [4pt] 
  $\quad$ $3v(j)\!=\!v(j')$ & $3^d$ \\ [4pt]

\hline

\vphantom{$\int^{a^{\int^a}}$} 
potentially supersingular, $F$ unramified &     \cr
$\quad$ $v(\Delta_\minn) < v(\Delta_\minn^\prime)$ & $1$  \\[4pt]
   $\quad$ $v(\Delta_\minn) > v(\Delta_\minn^\prime)$ & $3^d$ \\[4pt]

\hline
\vphantom{$\int^{a^{\int^a}}$}

(potentially) supersingular, $F$ ramified & $\,\,  1 \leq \alpha_{\phi,F} \leq 3^d$    \\ [4pt]
  
\noalign{\hrule height 1.2pt}

\noalign{\hrule height 1.2pt\phantom{$a_a$}}
\end{tabular}
\end{center}
\caption{The value of $\alpha_{\phi,F}$ when $F$ is a 3-adic field.}
\label{tab:values of alpha}

\end{table}

\begin{proof}
That $c(E')/c(E)$ lies in the set $\{1/3, 1,3\}$ follows from Tate's algorithm.  The quantity $\frac{\phi^* \omega^\prime}{\omega}$ is, up to a unit, the leading term in the power series giving the induced isogeny on formal groups $\hat E \to \hat E'$, and is therefore integral over $\Q_3$ \cite[4.2]{DD}, so that $\alpha_{\phi,F} \geq 1$.  Since $\alpha_{\phi,F} \alpha_{\hat\phi,F} = |3|^{-1} = 3^d$, we must have $\alpha_{\phi,F} \leq 3^d$ as well.   

Most of the entries in the two tables can be found in \cite[Table 1]{DD}.  The value of $\alpha_{\phi,F}$ in the case of potentially supersingular reduction over an unramified field $F$ is a result of the second author and Gealy \cite{kg}.    
\end{proof}

We are also interested in the behavior of the Selmer ratio $c(\phi)$ under twisting.  By Proposition \ref{localratio}, we will therefore want to know how the Kodaira type of $E$ changes under twisting.  If $p > 2$, then there are four classes in $F^*/F^{*2}$, and a set of representatives is $\{1, u, \pi, \pi u\}$, where $u$ is a non-square unit in $\O_F$.  The Kodaira type depends only on the base change of $E$ to the maximal unramified extension of $F$, so $E$ and $E_u$ have the same type, and $E_\pi$ and $E_{\pi u}$ have the same type as well.  For the twist by $\pi$, we have the following result.   

\begin{proposition}
\label{prop:Kodrel}
For $p \ge 3$, the relationship between the Kodaira types of $E$ and $E_\pi$ is given by Table $\ref{tab:twisted Kodaira symbols}$.
\end{proposition}
\begin{proof}
See \cite[Prop.\ 1]{comalada}, for example.
\end{proof}

\begin{table}

\begin{tabular}{@{\vrule width 1.2pt}c@{\ \vrule width 1.2pt\ }c|c|c|c|c| c|c|c|c|c@{\ \vrule width 1.2pt}}
\noalign{\hrule height 1.2pt}

$\quad$ \vphantom{$\int^{\int^\int}$} Kodaira type of $E$ $\quad$ & $\IZ$ & $\In{k}$ & $\II$ & $\III$ & $\IV$ & $\IZS$ & $\InS{k}$ & $\IVS$ & $\IIIS$ & $\IIS$ \\ [4pt] \hline

$\quad$ \vphantom{$\int^{\int^\int}$} Kodaira type of $E_\pi$ $\quad$ & $\IZS$ & $\InS{k}$ & $\IVS$ & $\IIIS$ & $\IIS$ & $\IZ$ & $\In{k}$ & $\II$ & $\III$ & $\IV$ \\ [4pt]
\noalign{\hrule height 1.2pt\phantom{$a_a$}}
\end{tabular}
\caption{The effect of twisting on Kodaira symbols for $p > 2$.}
\label{tab:twisted Kodaira symbols}
\end{table}

Things are more complicated when $p = 2$, but we still have the following result:
\begin{proposition}\label{twist1}
Suppose $p \neq 3$.  Then there exists $s \in F^*$ such that $c(\phi_s) = 1$.  
\end{proposition}

\begin{proof}
For $p > 3$, this follows from Propositions~\ref{localratio} and \ref{prop:Kodrel}.  For $p = 2$, this follows from \cite[Table~1]{comalada} and Proposition~\ref{localratio}.  
\end{proof}

Proposition \ref{twist1} can fail when $p = 3$, but a rather weaker statement along these lines is still true and will be proved as Corollary~\ref{cong} below.  While Corollary~\ref{cong} looks innocuous enough, it is crucial in the proof of Theorem~\ref{totreal}.  We first need a few lemmas.  

\begin{lemma}
\label{lem:prodalpha}
If $p = 3$, then $\alpha_{\phi,F}, \alpha_{\hat\phi,F} \in \Z$ and $\alpha_{\phi,F} \cdot \alpha_{\hat\phi,F} = 3^d$.
\end{lemma}
\begin{proof}
This follows from Theorem \ref{3adic} and its proof.
\end{proof}

\begin{lemma}\label{lem:euler}
For any prime $p$, we have
\[
\big | H^1(F,E[\phi]) \big | 
= \left \{ \begin{matrix} \big | E(F)[\phi] \big|\cdot\big | E^\prime(F)[\hat\phi] \big| & \text{ if }p \ne 3 \\ 3^d\cdot \big | E(F)[\phi] \big|\cdot \big | E^\prime(F)[\hat\phi] \big| & \text{ if }p = 3 \end{matrix}\right . \;.
\]
\end{lemma}
\begin{proof}
This is the local Euler characteristic formula, as $E[\phi]$ and $E^\prime[\hat\phi]$ are dual to each other.  
\end{proof}

\begin{corollary}
\label{cor:prodcokers}
For any prime $p$, we have
\[
\big | \coker E(F) \xrightarrow{\phi} E^\prime(F) \big |\cdot \big |  \coker E^\prime(F) \xrightarrow{\hat \phi} E(F) \big | =  \big | H^1(F,E[\phi]) \big |.
\]
\end{corollary}
\begin{proof}
We apply Lemma \ref{lem:tamag} to see that
\[
\frac{\big | \coker E(F) \xrightarrow{\phi} E^\prime(F) \big |}{\big| E(F)[\phi] \big|} \cdot \frac{ \big |  \coker E^\prime(F) \xrightarrow{\hat \phi} E(F) \big |}{\big | E^\prime(F)[\hat\phi]\big |} = \frac{c(E^\prime)}{c(E)}\alpha_{\phi,F}\frac{c(E)}{c(E^\prime)}\alpha_{\hat\phi,F} = \alpha_{\phi,F}\alpha_{\hat\phi,F}.
\]
The result now follows from Corollary \ref{cor:tamag} and  Lemma \ref{lem:prodalpha}.
\end{proof}
We now return to the case $p = 3$.  
\begin{corollary}\label{cong}
If $p = 3$, then there exists $s \in F^*$ such that $1 \leq c(\phi_{s}) \leq 3^d$, where $d = [F \colon \Q_p]$. 
\end{corollary}
\begin{proof}
The points of $E[\phi](\bar F)$ are defined over some (possibly) quadratic field $F(\sqrt D)$, while the points of the dual group $E[\hat\phi](\bar F)$ are defined over $F(\sqrt{-3D})$.  Choose $s$ not in the squareclass of either $D$ or $-3D$.  We 
then have ${E_{s}(F)[\phi]=0}$ and $E'_{s}(F)[\hat \phi] = 0$, so that $c(\phi_{{s}}) = \big | \coker E_{s}(F) \to E'_{s}(F) \big |$ and $c(\hat\phi_{{s}}) =  \big |  \coker E'_{s}(F) \to E_{s}(F) \big | $. 
By Lemma \ref{lem:euler}, we have $\big |H^1(F,E_{s}[\phi]) \big | = 3^d$, so by Corollary \ref{cor:prodcokers}, we obtain $c(\phi_{{s}}) \cdot c(\hat\phi_{{s}}) = 3^d$. Since $c(\phi_{{s}})$ and $c(\hat\phi_{{s}})$ are both integers, the desired conclusion follows.
\end{proof}

\section{Further applications to ranks of elliptic curves} \label{sec:appl-to-rank}
In this section, we prove Theorems \ref{totreal} and \ref{cm}.

\begin{proof}[Proof of Theorem $\ref{totreal}$]
We begin by assuming that $F$ is totally real of degree $d$.  Our first aim is to show that $T_0(\phi)$ is non-empty.  For any $s \in F^*$, we define $c_3(\phi_s) = \prod_{\p \mid 3} c_\p(\phi_s)$, parallel to the definition (\ref{cinftyphis}) of $c_\infty(\phi_s)$.  We now use Proposition \ref{twist1}, Corollary \ref{cong}, and weak approximation to find $s \in F^*$ so that
\begin{enumerate}
\item $c_\p(\phi_s) = 1$ for all $\p \nmid 3\infty$,
\item $1 \leq c_3(\phi_s) \leq 3^d$, and
\item $c_\infty(\phi_s) = 1/c_3(\phi_s)$.
\end{enumerate}  
The final condition is possible because for each infinite place $\p$, $c_\p(\phi_s)$ is equal to $1$ or $1/3$ and both possibilities do occur for each of the $d$ real embeddings of $F$.  This shows that the set $T_0(\phi)$ is non-empty and hence has positive density.  The proof that $T_1(\phi)$ has positive density is similar: this time we choose $s$ so that $c_\infty(\phi_s) \cdot c_3(\phi_s) \in \{3, 1/3\}$, which is always possible.  Theorem \ref{totreal}$(i)$ now follows from Theorem \ref{globalbounds}  

Now suppose that $F$ has degree $d$ and exactly one complex place.  As before, we can find $s \in F^*$ satisfying conditions (1) and (2) above, but now the function $c_\infty(\phi_s)$ satisfies $\frac{1}{3^{d-1}} \leq c_\infty(\phi_s) \leq \frac{1}{3}$.  All possible values in this range occur for some $s$, even if we impose conditions (1) and (2).  If $d > 2$, then we can choose $c_\infty(\phi_s)$ so that $c(\phi_s) = c_\infty(\phi_s)c_3(\phi_s) \in \{3,1/3\}$ and hence $T_1(\phi)$ is non-empty.  If $d = 2$, then we can only guarantee that $T_1(\phi) \cup T_0(\phi)$ is non-empty using this method.

We postpone the proof of the final statement of the theorem until Section \ref{examples}, where it will follow from Propostion \ref{prop:3 Selmer lower bd}.
\end{proof}

To prove Theorem \ref{cm}, we begin with a lemma.

\begin{lemma}\label{twist}
Suppose $p$ is a prime, $L$ is a finite extension of $\Q_p$, and $\phi \colon E \to E'$ is a $3$-isogeny over $L$.  Assume that there is a quadratic twist of $E$ with good reduction.  Then $c(E) = c(E')$.
\end{lemma}

\begin{proof}
We may assume that $E$ has additive, potentially good reduction.  If $p > 3$, then the lemma follows from Theorem  \ref{integrality}. If $p = 2$, then by \cite[Lem.\ 4.6]{DD}, we have
\[\dfrac{c(E')}{c(E)} = \left|\dfrac{E'(L)}{NE'(M)}\right|\left|\dfrac{E(L)}{NE(M)}\right|^{-1},\]
where $M$ is a quadratic extension of $L$ and $N \colon E(M) \to E(L)$ is the norm.  As the left hand side is a power of 3 and the right hand side is a power of 2, we conclude that $c(E')= c(E)$.  If $p = 3$, we use a similar argument in conjunction with \cite[Lem.\ 4.7]{DD}.   
\end{proof}

The following result will imply Theorem \ref{cm}.   

\begin{theorem}\label{cm1}
Suppose $\phi \colon E \to E'$ is a $3$-isogeny of elliptic curves over a number field $F$ and $\End(E) \simeq \End(E') \simeq \O_K$ for some imaginary quadratic field $K$.  Then $c(\phi) = 1$.  
\end{theorem}  

\begin{proof}
We first show that for all finite places $\p$ of $F$, $c_\p(E') = c_\p(E)$.  By \cite[5.22]{rubin}, some twist of $E$ has good reduction at $\p$.  If this twist is a quadratic twist, then it follows from Lemma \ref{twist} that $c_\p(E') = c_\p(E)$.  If the twist is not a quadratic twist, then we must have $j(E) \in \{0, 1728\}$.  The case $j(E) = 1728$ does not occur because that is the unique $j$-invariant with CM by $\Z[i]$ and these elliptic curves do not admit endomorphisms of degree~3.  If $j(E) = 0$, then we must have $j(E') = 0$ since there is a unique $j$-invariant with CM by the ring of integers of $\Q(\sqrt{-3})$.  Since $F$ necessarily contains $\Q(\sqrt{-3})$, we even have $E \simeq E'$ and $\phi$ is (up to automorphism) multiplication by $\sqrt{-3}$.  Hence $c_\p(E') = c_\p(E)$ in all cases.  

In particular, if $\p$ is a finite place not above 3, then $c_\p(\phi) = 1$ by Lemma \ref{lem:tamag}.  On the other hand, if $\p$ is an infinite place of $F$, then $\p$ is complex since $K \subset F$, and so $c_\p(\phi) =  \frac{1}{3}$ by Proposition \ref{R}.  So it remain to show that $c_3(\phi) := \prod_{\p | 3} c_\p(\phi)$ is equal to $3^{[F \colon K]}$.  By Lemma \ref{lem:tamag}, we have $c_3(\phi) = \prod_{\p | 3} \alpha_{\phi, F_\p}$, where recall $\alpha_{\phi, F_\p}$ is the inverse absolute value of the derivative at the origin of the isogeny of formal groups induced by $\phi$.    

Since $\End(E) \simeq \End(E') \simeq \O_K$, we may choose a prime $\ell \neq 3$ and an $\ell$-isogeny $\lambda \colon E' \to E$.  The composition $\lambda \circ \phi$ is an endomorphism $a  \in \End(E) \simeq \O_K$, and we have $\alpha_{\phi, \p} \alpha_{\lambda,\p} = \alpha_{a, \p}$.  But $\alpha_{\lambda,\p} = 1$ since $\ell \neq 3$, and $\alpha_{a, \p} = |a|_\p^{-1}$ by \cite[Prop.\ 3.14]{rubin}.  Thus $\alpha_{\phi,\p} = |a|_\p^{-1}$ and we compute
\[c_3(\phi) = \prod_{\p | 3} \alpha_{\phi, \p} = \prod_{\p | 3} |a|_\p^{-1} = \left(\prod_{v | 3} |a|_v^{-1}\right)^{[F \colon K]},\]
where the final product is over places $v$ of $K$ above 3.  Since $\prod_{v | 3} |a|_v^{-1}$ is the 3-part of 
\[\Nm_{K/\Q}(a) = \deg(a \colon E \to E) = 3\ell,\] 
we conclude that $c_3(\phi) = 3^{[F\colon K]}$, as desired.
\end{proof}

\begin{remark}
{If either $\End(E)$ or $\End(E')$ is a non-maximal  order in $K$, then it can happen that $c(\phi) \neq 1$.}
\end{remark}

\begin{proof}[Proof of Theorem $\ref{cm}$]
First suppose that $\End(E) \simeq \O_K$.  Since 3 is not inert in $\O_K$, there is a 3-isogeny $\phi \colon E \to E'$ with $\End(E') \simeq \O_K$.  Indeed, we may take $E' = E/E[\p]$ where $\p \subset \O_K$ is an ideal of norm 3 and $E[\p]$ is the $\p$-torsion on $E$.  Note that $\End(E_s) \simeq \End(E'_s) \simeq \O_K$ for any $s \in F^*/F^{*2}$.  It follows from Theorem \ref{cm1} that $c(\phi_s) = 1$ for all $s \in F^*/F^{*2}$, or in other words, $F^*/F^{*2} = T_0(\phi)$.  We conclude from Theorem \ref{globalbounds} that the average rank of $E_s(F)$ is at most 1, and that at least $50\%$ of twists $E_s$ have rank 0.

In the general case where $\End(E)$ is just an order in $\O_K$, we may choose an elliptic curve $\tilde E$ with $\End(\tilde E) \simeq \O_K$ which is isogenous to $E$ \cite[Prop.\ 5.3]{rubin}.  Then for each $s \in F^*/F^{*2}$, the quadratic twist $E_s$ is isogenous to $\tilde E_s$.  Theorem \ref{cm} for $E$ therefore follows from Theorem \ref{cm} for $\tilde E$, already proven in the previous paragraph.        
\end{proof}

\section{Ranks of Jacobians of trigonal curves}\label{superelliptic}
In this section we describe a class of Jacobians satisfying the hypothesis of Theorem \ref{ab var ranks}.  We work over a number field $F$ and consider trigonal curves of the form
\[C\colon y^3  = f(x),\]
where $f(x)$ is a monic polynomial of degree $d \geq 3$ over~$F$, with no repeated roots over $\bar F$.  If $d > 3$, then $C$ is singular at the unique point at infinity.  Suppose also that $3 \nmid d$.  Then there is a unique point $\infty$ in the normalization lying above the point at infinity.  The normalization, which we will also call $C$, has genus $g = d -1$, by Riemann-Hurwitz.  

Embed $C$ in its Jacobian $J$ using the rational point $\infty$.  Let $a_1, \ldots, a_d$ be the roots of $f$ over $\bar F$, and write $P_i$ for the point $(a_i,0)$ on $C$.  Since the divisor of the function $x - a_i$ is $3P_i - 3\infty$, we see that each $P_i$ is 3-torsion in $J(\bar F)$.  In fact, the points $P_i$ are fixed by the automorphism $\zeta \colon (x,y) \mapsto (x,\zeta_3 y)$, where $\zeta_3$ is a primitive third root of unity, and so they lie in the kernel of the endomorphism $\psi = 1 - \zeta$ of $J$.  This endomorphism is defined only over $F(\zeta_3)$, but its kernel descends to a subgroup scheme $G$ of $J$ of order $3^g$ defined over~$F$.  By \cite[Prop. 3.2]{schaefer2}, the points $P_1, \ldots, P_d$ span $G(\bar F)$ with the unique relation $\sum_{i = 1}^d P_i = 0$.  

Over the field $K = F(\zeta_3)$ we have $(J/G)_K  \simeq J_K/J_K[\psi] \simeq J_K$, so $J/G$ is a twist of $J$. In fact:

\begin{proposition}\label{ppav}
$J/G$ is the quadratic twist $J_{-3}$ of $J$ by the quadratic field $K = F(\zeta_3)$.  
\end{proposition}

\begin{proof}
To compute the cocycle $g \colon \Gal(K/F) \to \Aut(J_K)$ giving the twisting data for $J/G$, we use the isomorphism $f \colon J_K \to (J/G)_K$ that sends a point $P$ to the class of any $Q$ such that $\zeta(1-\zeta)(Q) = P$.  This is indeed an isomorphism over $K$ with inverse $P \mapsto \zeta(1-\zeta)(P)$.  The corresponding cocycle $g \colon \Gal(K/F) \to \Aut(J_K)$ is by definition $g(\sigma) = f^{-1}f^\sigma$.  It is determined by $g(\tau)$ where $\tau$ is the generator of $\Gal(K/F)$.  We have $g(\tau)(P) = f^{-1}(Q)$ where $Q$ is such that $\zeta(\zeta -1)(Q) = P$.  Thus 
\[g(\tau)(P) = f^{-1}(Q) = \zeta(1-\zeta)(Q) = -P.\]  
So $g(\tau) = -1$, which shows that $J/G$ is the quadratic twist of $J$ by $K$.        
\end{proof}

\begin{corollary}
The abelian variety $J/G$ is principally polarized.
\end{corollary}

\begin{proof}
By \cite[Lem.\ 4.16]{morgan}, the quadratic twist of a principally polarized abelian variety is itself principally polarized.
\end{proof}

\begin{theorem}\label{superelliptic twists}
Suppose that $3\nmid d$ and that at least $d-2$ roots of $f(x)$ lie in $F$.  Then the average rank of the quadratic twists $J_s$, $s \in F^*/F^{*2}$, is bounded. 
\end{theorem}
\begin{proof}
Use the principal polarizations on $J$ and $J/G$ to identify these abelian varieties with their duals.  Then the composition
\[J \stackrel{\psi}{\longrightarrow} J/G \stackrel{\hat\psi}{\longrightarrow} J, \]
is the map $[3]:J\to J$, up to an automorphism of $J$ over~$F$.  As there are at least $d-2 = g - 1$ rational roots of $f(x)$, we may choose $g-1$ of the points $P_i$ which are rational.  Since the unique relation among all the $P_i$'s is $\sum_{i = 1}^{g+1} P_i = 0$, the chosen rational points $P_i$ generate a subgroup $H \simeq (\Z/3\Z)^{g-1}$ of $G$.  We may then factor $\psi$ as 
\[J \to J/H \to J/G.\]  The first map is a composition of 3-isogenies and the second is itself a 3-isogeny.  Thus, $\psi$ factors as a composition of $3$-isogenies over~$F$, and by duality, so does $\hat\psi$. It follows that $[3]\colon J \to J$ factors as a composition of 3-isogenies.  So our Theorem 1.2 applies, and the average rank of the quadratic twists $J_s$, $s\in F^*/F^{*2}$, is bounded.  
\end{proof}

\begin{corollary}\label{cor:super}
Let $F$ be a number field.  For every positive integer $g$ not congruent to $2\pmod 3$, there exists a Jacobian $J$ of genus $g$ over~$F$ such that the average rank of $J_s$, $s \in F^*/F^{*2}$, is bounded. 
\end{corollary}

Many of the Jacobians considered by Theorem \ref{superelliptic twists} and Corollary \ref{cor:super} are absolutely simple.  For example, the Jacobian of the curve $y^3 = x^p - 1$ for some prime $p$ is absolutely simple as it has a primitive CM type \cite{Hazama}.  As such, for $100\%$ of primes $q \equiv 1 \pmod{3p}$, its reduction modulo $q$ is also absolutely simple by \cite[Theorem 3.1]{MP}.  Thus, if $f(x)$ is congruent to $x^p-1$ modulo such a prime $q$, then the Jacobian of the curve $y^3 = f(x)$ must also be absolutely simple.  We thank Ben Howard for raising the curves $y^3 = x^p - 1$ to our attention.

It is possible to state a more precise version of Theorem 11.3, bounding the average rank in terms of the  averages of certain local Selmer ratios; the latter should be computable in principle by extending the methods of \cite{SW2}.
If $J$ additionally has complex multiplication, then we can sometimes make our bound on the average rank very explicit.  In some cases, the bound is even strong enough to imply that a positive proportion of twists have rank 0.  We illustrate this with a particularly pretty example.

\begin{theorem}\label{picard50}
Let $J$ be the base change of the Jacobian of the Picard curve $C \colon y^3 = x(x^3 - 1)$ to the field $L = \Q(\zeta_9)$.  Then the average rank of the quadratic twists $J_s$, $s \in L^*/L^{*2}$, is at most $3$.  Moreover, at least $50\%$ of these twists have rank $0$.  
\end{theorem}

\begin{proof}
The curve $C$ falls under the rubric of Theorem \ref{superelliptic twists}.  But it also admits an order 9 automorphism $\zeta_9 \colon (x,y) \mapsto (x\zeta_9^3, y\zeta_9)$ whose cube is what we were calling $\zeta$.  Thus, $J$ has CM by the cyclotomic ring $\O_L = \Z[\zeta_9]$.  The endomorphism $\phi \:=1 - \zeta_9$ has degree~3, and its kernel $J[\phi]$ is generated by the rational point $(0,0)$.  Since $3\O_L = (1 - \zeta_9)^6\O_L$, we have $[3] = u \circ \phi^6$ for some unit $u \in \O_L^*$.  For any $s \in L^*/L^{*2}$, the quadratic twist $J_s$ also has CM by $\O_L$, and $\phi_s$ is again (up to a unit) the endomorphism $1 - \zeta_9$ on $J_s$.  

Let $\p = (1 - \zeta_9)\O_L$ be the unique prime of $L$ above $3$.  Then $c_\mathfrak{q}(\phi_s) = c_\mathfrak{q}(J_s)/c_\mathfrak{q}(J_s) = 1$ for all finite places $\mathfrak{q} \neq \p$.  At the prime $\p$, we have $c_\p(\phi_s) = |\alpha_s|_\p^{-1}$, where $\alpha_s$ is the determinant of the derivative of the endomorphism $\phi^*_s$ induced by $\phi_s$ on the formal group.  Thus, $c_\q(\phi_s)^6 = c_\q([3]) = |27|_\p^{-6} = 3^{18}$, and hence $c_\q(\phi_s) = 27$.  We conclude that the global Selmer ratio is equal to 
\[c(\phi_s) = \prod_{\q \leq \infty} c_\q(\phi_s) = \left(\prod_{\q \mid \infty} \dfrac13\right) c_\p(\phi_s) = \left(1/3\right)^3 \cdot 27 = 1.\] 
By Theorem 1.2, the average size of $\Sel_\phi(J_s)$ is equal to $1 + 1 = 2$.  The inequality $2r + 1 \leq 3^r$ then shows that the average rank of $\Sel_\phi(J_s)$ is at most 1/2.  Since $[3] = u \circ \phi^6$, we have 
\[\avg \, \dim_{\F_3}\Sel_3(J_s) \leq 6\cdot \avg\, \dim_{\F_3}\Sel_\phi(J_s) \leq 3,\] and so the average rank of $J_s$ is at most 3, proving the first part of the theorem.  

For the second part, note that the rank of $J_s$ is divisible by 6, since $J_s(L)$ admits a faithful action of $\O_L$.  So if less than $50\%$ of twists have rank 0, then more than $50\%$ have rank 6, which forces the average rank to be above 3, contradicting the first part of the theorem.    
\end{proof}

\begin{lemma}
The abelian variety $J$ is absolutely simple.
\end{lemma}

\begin{proof}
Since $J$ has CM by $L = \Q(\zeta_9)$, it suffices by \cite[Thm.\ 3.5]{Lang} to show that its CM-type is not induced from the subfield $\Q(\zeta_9^3) = \Q(\zeta_3)$.  The endomorphism $\zeta_9 \in \End(J)$ acts on the holomorphic form $dx/y^2$ by multiplication by $\zeta_9 \in L$, whereas it acts on the holomorphic form $dx/y$ by multiplication by $\zeta_9^2$.  In particular, the endomorphism $\zeta_9^3 = \zeta_3$ does not act on all holomorphic forms via the same embedding into $L$, and so the CM-type is not induced from $\Q(\zeta_3)$.   
\end{proof}

Thus, Theorem \ref{picard50} gives the first example of an absolutely simple abelian variety of dimension $g >1$ over a number field such that a positive proportion of its quadratic twists have rank 0. 
In \cite{ari}, these methods are extended to in fact give many examples over $\Q$.

\section{Examples}\label{examples}
We first give the examples referred to after Theorem \ref{cm} in the Introduction.
\begin{example}{\em 
Let $E$ be defined by $y^2 = x^3 + x^2 - 114x - 127$ (Cremona label 196b1), with $3$-isogeny $\phi \colon E \to E'$. Over $\Q$, $E$ has type $\IV$ reduction at $p = 2$, type $\IVS$ reduction at $p = 7$.  At $p = 3$, $E$ has good ordinary reduction with $\alpha_{\phi,\Q_3} = 1$.  

Let $F$ be any imaginary quadratic field in which $2$ is inert 
and $E_F$ the base change of $E$ to $F$. Since $2$ inert in $F$, we have $\mu_3 \subset F_\p$ for all primes $\p$ at which $E$ has bad reduction. By Proposition \ref{localratio}, we therefore have $c_\p(E_{F,s}') / c_\p(E_{F,s}) = 1$ for all primes $\p$ and for all $s \in F^*$.

We therefore have $c(\phi_s) = \alpha_s / 3$, where $\alpha_s = \prod_{\p | 3} \alpha_{\phi_s, F_\p}$. Since $E$ comes from a base change of a curve over $\Q$, Proposition 4.7.(2) in \cite{DD} combined with item (i) above tells us that $\alpha_{\phi, F_\p} = 1$ for each $\p \mid 3$. Further, since $E$  has good ordinary reduction at each $\p \mid 3$, we have  that $\alpha_{\phi_s,F_\p} = \alpha_{\phi, F_\p} = 1$.  We conclude that $c(\phi_s) = 1/3$ for all $s \in F^*$. This shows $T_1(\phi) = F^*/F^{*2}$, and so $5/6$ of twists $E_{F,s}$ have 3-Selmer rank 1, by Theorem \ref{globalbounds}(c). 
}
\end{example}

\begin{example}{\em
We give an example of a non-CM curve satisfying the conclusion of Theorem \ref{cm}.  Let $F = \Q(\sqrt{-3})$ and let $E$ be the elliptic curve 
\[E \colon y^2 = x^3  + \frac{9}{4}\left(  1 + \sqrt{-3}\right)^2  (\sqrt{-3}x - 4)^2.\] 

The curve $E$ has bad reduction only at the unique prime $\p_3$ above $3$.
While $E$ has potentially supersingular reduction and reduction type $\IVS$ at $\p_3$, $E$ has the peculiar property that $\alpha_{\phi_s,F_{\p_3}} = 3$ and $c_{\p_3}(E'_s) = c_{\p_3}(E_s)$ for every $s \in F^*$. This fact does not directly follow from any of the results in this paper, but it can be observed computationally or deduced from the fact that $E$ is a $\Q$-curve.

We therefore find that $c(\phi_s) = 1$ for all $s \in F^*$.  By Theorem \ref{avgrank}, the average rank of $E_s$ is at most 1, and by  Theorem \ref{globalbounds}, at least 50\% of these twists have rank 0.
} 
\end{example}

Our third and final example will be used in our proof of the final assertion of Theorem \ref{totreal}. 
\begin{example}
\label{ex:infinite families}
{\em
For any $k$, let $t_k = 4^k \cdot 3^8$ and define $E_k$ to be the curve given by 
\[E_k \colon y^2 = x^3 + ( t_k + 27)  ( x + 4 )^2.\] The curve $E_k$ has a $3$-isogeny $\phi:E_k \rightarrow E_k^\prime$ with $\ker \phi = \langle (0, 2\sqrt{t_k+27})\rangle$. We claim that if $F$ is a number field, then $t(\phi_s) \ge r_2$ for all $s \in F^*/F^{*2}$, where $r_2$ is the number of complex places of $F$.

A discriminant calculation shows that $E_k$ has good reduction away from places dividing $6 (t_k + 27)$. Since $j(E_k)  = \frac{(t_k + 3)^3 (t_k + 27)}{t_k}$, $E$ has additive reduction at all bad places except those dividing $6$.  Suppose that $\p$ is a prime of $F$ dividing a rational prime $p > 3$.  If $E$ has bad reduction at $\p$, then $\p \mid ( t_k + 27)$, and as a result, $t_k \equiv -27 \pmod \p$. Since $t_k$ is manifestly a square, we see that $-3 \in F_\p^{*2}$. Since $E$ has additive reduction at $\p$ and $ \mu_3 \subset F_\p$, Proposition \ref{localratio} tells us that $c_\p(\phi_s) = 1$ for all $s \in F_\p^*$.

We now turn our attention to primes dividing $2$ and $3$.  If $\p \mid 2$, then we observe that $\p$ formally divides the denominator of $j(E_k)$ but not the numerator. We therefore see that $E$ has (potential) multiplicative reduction at $\p$. The $j$-invariant $j(E^\prime)$ is given by $\frac{(t_k + 27) (t_k + 243)^3}{t_k^3}$, and so $v(j^\prime) = 3v(j)$. Proposition \ref{localratio} then tells us that $c_\p(\phi_s) \in \{1,3\}$ for all $s \in F_\p^*$.

If $\p \mid 3$, then $\p$ formally divides both the numerator and denominator of $j(E_k)$. We find however that since $t_k$ is sufficiently divisible by $3$, we have $v(j) = - 2v(3)$.  Similarly, we have $v(j^\prime) = -6 v(3)$.  Combining Proposition \ref{localratio} and Theorem \ref{3adic}, we obtain $c_\p(\phi_s) \in \{3^{[F_\p:\Q_3]}, 3^{[F_\p:\Q_3]+1}\}$ for all $s \in F_\p^*$.

We therefore have
\begin{multline}
\label{eq:Tbd}
c(\phi_s) = 3^{-r_2} \prod_{v \textrm{ real} } c_v(\phi_s) \prod_{\p \nmid \infty} c_\p(\phi_s)
\ge 3^{-(r_1 + r_2)} \prod_{\p \nmid \infty} c_\p(\phi_s)
\\
\ge 3^{-(r_1 + r_2)} \prod_{\p | 6} c_\p(\phi_s)
\ge 3^{-(r_1 + r_2)} \prod_{\p | 3} 3^{[F_\p:\Q_3]}
= 3^{-(r_1 + r_2)} \cdot 3^{[F:\Q]}
= 3^{r_2}.
\end{multline}
}
\end{example}

As a result, we obtain the following:

\begin{proposition}
\label{prop:3 Selmer lower bd}
If $E = E_k$ as in Example $\ref{ex:infinite families}$, then $\dim_{\mathbb{F}_3} \Sel_3(E_s) \ge r_2$ for all $s \in F^*/(F^*)^2$.
\end{proposition}
\begin{proof}
It follows immediately from \eqref{fiveterm} and \eqref{cassels formula} that if $E = E_k$, then $\dim_{\mathbb{F}_3} \Sel_3(E_s) \ge r_2$ except for the sole $s \in F^*/(F^*)^2$ for which $E_s^\prime(F)[\hat \phi] \ne 0$. When $E_s^\prime(F)[\hat \phi] \simeq \Z /3\Z$, we need to show that $c(\phi_s) \ge 3^{r_2 + 2}$. We do so by showing that $E_{s}$ has split multiplicative reduction at all primes $\p \mid 6$.

By direct calculation, we have $-c_6(E_{s}) = -1728 (t_k + 27)^4 (t_k^2 + 18 t_k - 27)$. In $\Q^*/\Q^{*2}$, we therefore have
\begin{multline*}
-c_6(E_{s}) = -3 (t_k^2 + 18 t_k - 27) = -3(4^{2k} \cdot 3^{16} + 4^{k+1}\cdot 3^{10} - 27) 
\\
= 3^4(1 - 4^{2k} \cdot 3^{13} - 4^{k+1}\cdot 3^{7}) = (1 - 4^{2k} \cdot 3^{13} - 4^{k+1}\cdot 3^{7}).
\end{multline*}
As $(1 - 4^{2k} \cdot 3^{13} - 4^{k+1}\cdot 3^{7})$ is visibly a square in both $\Q_2$ and $\Q_3$, we see that $-c_6(E_{s})$ is a square in $F_\p^*$ for all $\p \mid 6$. As a result, $E_{s}$ has split multiplicative reduction at all primes $\p \mid 6$ and by  Proposition \ref{localratio}, we therefore have $c_\p(E'_{s})/c_\p(E_{s}) = 3$.

Recalculating the bound from in \eqref{eq:Tbd}, we therefore have 
\begin{equation*}
c(\phi_{s}) = 3^{r_2}  \prod_{\p | 6} c_\p(\phi_s) =  3^{r_2}  \cdot 3^{|\{ \p : \p \mid 6\}|} \ge 3^{r_2 +2},
\end{equation*}
where the final inequality follows from the fact that both $2$ and $3$ must each have at least one prime above them in $F$. It then follows from \eqref{fiveterm} and \eqref{cassels formula} that $\dim_{\mathbb{F}_3} \Sel_3(E_s) \ge r_2$.
\end{proof}

\begin{remark}{\em 
Proposition \ref{prop:3 Selmer lower bd} is an analogue of a similar result for $2$-Selmer groups obtained by the second author for curves having a rational point of order two \cite{KlagsbrunLowerBound}.}
\end{remark}

\begin{remark}{\em 
The technique used to construct the family $\mathcal{E}$ in Example $\ref{ex:infinite families}$ can be use to construct infinite families $\mathcal{E}_p$ for each $p \in \{ 3, 5, 7, 13\}$ such that for each $E \in \mathcal{E}_p$, $E$ has a cyclic $p$-isogeny and $\dim_{\mathbb{F}_p} \Sel_p(E_s) \ge r_2$ for every $s \in F^*$.
} 
\end{remark}

\section{Acknowledgements}
We thank N.\ Elkies, B.\ Howard, E.\ Howe, M.\ Stoll, A.\ Sutherland, and J.\ Wetherell for helpful comments and discussions.
MB was partially supported by a Simons Investigator Grant and NSF Grant DMS-1001828.
RJLO was partially supported by NSF grant DMS-1601398.

\bibliographystyle{abbrv}
\bibliography{references}

\end{document}